\documentclass[english,12pt,a4paper,reqno]{amsart}
\usepackage{amsmath,amsthm,amsfonts,amssymb,amscd}
\usepackage{url}
\usepackage[ansinew]{inputenc}
\usepackage{graphicx}
\usepackage{a4wide}
\usepackage{subfigure}
\usepackage{graphics}
\usepackage{graphicx}
\usepackage{hyperref}
\usepackage{xfrac}
\usepackage{array}
\usepackage[all]{xy}
\usepackage{epic}
\usepackage{latexsym}
\usepackage{enumerate}
\usepackage{appendix}
\usepackage{mathrsfs}
\newtheorem{theorem}{Theorem}[section]
\newtheorem{proposition}[theorem]{Proposition}
\newtheorem{lemma}[theorem]{Lemma}

\newtheorem{maintheorem}{Theorem}

\newcommand{\cmt}{\begin{maintheorem}}
\newcommand{\fmt}{\end{maintheorem}}

\newtheorem{maincorollary}[maintheorem]{Corollary}

\newcommand{\cmc}{\begin{maincorollary}}
\newcommand{\fmc}{\end{maincorollary}}

\theoremstyle{remark}

\newtheorem{definition}[theorem]{Definition}
\newtheorem{remark}[theorem]{Remark}

\newtheorem{corollary}[theorem]{Corollary}

\def\cv{\ensuremath{\text {Cor}}}
\def\ld{\ensuremath{\text{LD}}}
\def\dist{\ensuremath{\text{dist}}}
\def\const{\ensuremath{\text{Const}}}
\def\d{\ensuremath{\text{d}}}

\def\vap{\varphi}

\begin{document}

\title[Statistical stability for Rovella maps]{Statistical stability and limit laws for Rovella maps}

\author{Jos\'e F. Alves}
\address{Jos\'e F. Alves\\ Departamento de Matemática, Faculdade de Ciências da Universidade do Porto\\
Rua do Campo Alegre 687, 4169-007 Porto, Portugal}
\email{jfalves@fc.up.pt} \urladdr{http://www.fc.up.pt/cmup/jfalves}

\author{Mohammad Soufi}
\address{Mohammad Soufi\\ Departamento de Matemática, Faculdade de Ciências da Universidade do Porto\\
Rua do Campo Alegre 687, 4169-007 Porto, Portugal}
\email{msoufin@gmail.com}


\thanks{The authors were partially supported by Funda\c c\~ao Calouste Gulbenkian, by the European Regional Development Fund through the programme COMPETE and by  FCT under the projects PEst-C/MAT/UI0144/2011 and PTDC/MAT/099493/2008.}



\begin{abstract}
We consider the family of one-dimensional maps arising from the contracting Lorenz attractors studied by Rovella.  Benedicks-Carleson techniques were used in \cite{rovella}  to prove that there is a one-parameter family of maps whose derivatives along their critical orbits increase exponentially fast and the critical orbits have slow recurrent to the critical point. Metzger proved in \cite{metzger} that these maps have a unique absolutely continuous ergodic invariant probability measure (SRB measure). 

Here we use the technique developed by Freitas in~\cite{freitas05,freitas} and show that the tail set (the set of points which at a given time have not achieved either the exponential growth of derivative or the slow recurrence) decays exponentially fast as time passes. As a consequence, we obtain the continuous variation (in the $L^1$-norm) of the densities of the SRB measures and associated metric entropies with the parameter.   Our main result also implies some statistical properties for these maps.
\end{abstract}

\maketitle


\section{Introduction}
The Theory of Dynamical Systems studies processes which evolve in time, and  whose evolution  is given by a flow or iterations of a given map. The main goals of this theory are to describe the typical behavior of orbits as time goes to infinity and to understand how this behavior changes when we perturb the system or to which extent it is stable.

The contributions of Kolmogorov, Sinai, Ruelle, Bowen, Oseledets, Pesin, Katok, Ma\~n\'e and many others turned the attention of the focus on the study of a dynamical system from a topological perspective to a more statistical approach and Ergodic Theory experienced an unprecedented development. In this approach, one tries in particular to describe the average time spent by typical orbits in different regions of the phase space. According to the Birkhoff's Ergodic Theorem, such averages are well defined for almost all points, with respect to any invariant probability measure. However, the notion of typical orbit is usually meant in the sense of volume (Lebesgue measure), which may not be invariant.

It is a fundamental open problem to understand under which conditions the behavior of typical points is well defined, from the statistical point of view. This problem can be precisely formulated by means of \emph{Sinai-Ruelle-Bowen (SRB) measures} which were introduced by Sinai 
for Anosov diffeomorphisms and later extended by Ruelle and Bowen for Axiom A diffeomorphisms and flows. 
Here we consider discrete time systems given by a map  $f:I\to I$ on an interval $I\subset \mathbb R$. Given an $f$-invariant Borelian probability $\mu$ in $I$, we call \textit{basin} of $\mu$ the set $B(\mu)$ of points $x\in I$ such that
\begin{equation}\label{chapequ}
\lim\limits_{n\to +\infty}\frac 1n\sum\limits_{j=0}^{n-1}\varphi(f^j(x))=\int\varphi~d\mu,\quad\mbox{for~any~continuous~}\varphi:I\to\mathbb{R}.
\end{equation}
We say that $\mu$ is an \textit{SRB measure} for $f$ if the basin $B(\mu)$ has positive Lebesgue measure in $I$. 

Trying to capture the persistence of the statistical properties of a dynamical system, Alves and Viana in \cite{alves4} proposed the notion of \textit{statistical stability}, which expresses the continuous variation of SRB measures as a function of the dynamical system in a given family $\mathcal F$ of maps endowed with some topology. Assume that each one of these maps in $\mathcal F$ has a unique SRB measure.
We say that $f$ is  \textit{statistically stable}  if  the map associating to each $g\in\mathcal F$ its SRB measure $\mu_g$ is continuous at $f$. Regarding the continuity in the space, we may consider weak* topology or even strong topology given by the $L^1$-norm in the space of densities (if they exist) with respect to Lebesgue measure.

Based on the work \cite{alves4},    sufficient conditions for the strong statistical stability of non-uniformly expanding maps were given in  in \cite{alves1}. The conditions have to do with the volume decay of the tail set, which is the set of points that resist satisfying either a non-uniformly expanding requirement or a slow recurrence, up to a given time. Freitas proved in \cite{freitas05, freitas} that the Benedicks-Carleson quadratic maps are non-uniformly expanding and slowly recurrent to the critical set, and the volume of their tail sets decays exponentially fast, so that the results in \cite{alves1} apply. Thus, these maps are statistically stable in the strong sense. For this purpose Freitas elaborated on the Benedicks-Carleson techniques in the phase space setting.

\subsection{Contracting Lorenz attractor}

The  geometric Lorenz attractor is the first example of robust attractor for a flow containing a hyperbolic singularity \cite{williams}. The singularity is accumulated by regular orbits which prevent the attractor to be hyperbolic. Lorenz attractor is a transitive maximal invariant set for a flow in $3$-dimensional spaces induced by a vector field having singularity at origin which the derivative of vector field at singularity has real eigenvalues $\lambda_2<\lambda_3<0<\lambda_1$ with $\lambda_1+\lambda_3>0$. 

%
%
The construction of the flow containing this attractor is the same as the geometric Lorenz flow. The original smooth vector field $X_0$ in $\mathbb{R}^3$ has the following  properties:
\begin{itemize}
\item[A1)] $X_0$ has a singularity at 0  for which the eigenvalues $\lambda_1,\lambda_2,\lambda_3\in \mathbb R$ of $DX_0(0)$ satisfy: 
\begin{itemize}
\item[i)]  $0<\lambda_1<-\lambda_3<-\lambda_2$,
\item[ii)]  $r>s+3$, where $r=-\frac{\lambda_2}{\lambda_1}$, $s=-\frac{\lambda_3}{\lambda_1}$;
\end{itemize}
\item[A2)]  there is an open set $U\subset \mathbb{R}^3$, which is positively invariant under the flow, containing the the cube $\{(x,y,z): |x|\leq 1, |y|\leq 1, |z|\leq 1\}$. The top of the cube is a Poincar\'e section foliated by stable lines $\{x=\const\}\cap\Sigma$ which are invariant under Poincar\'e first return map $P_0$. The invariance of this foliation uniquely defines a one dimensional map $f_0:I\setminus\{0\}\to I$ for which
$$f_0\circ \pi=\pi\circ P_0,$$
where $I$ is the interval $[-1,1]$ and $\pi$ is the canonical projection $(x,y,z)\mapsto x$; 
\item[A3)]  there is a small number $\rho>0$ such that the contraction along the invariant foliation of lines $x=\const$ in $U$ is stronger than $\rho$.
\end{itemize}
Observe that Rovella replaced the usual expanding condition $\lambda_3+\lambda_1>0$ in the hyperbolic singularity of the Lorenz flow  by the contracting condition $\lambda_3+\lambda_1<0$. The one-dimensional map $f_0$ satisfies the following properties:
\begin{enumerate}
\item[B1)] $f_0$ has a discontinuity at $x=0$ and 
$$\lim_{x\to 0^+}f_0(x)=-1,\quad 
 \lim_{x\to 0^-}f_0(x)=1;$$
\item[B2)]  $f'_0(x)>0$ for all $x\neq 0$ with $\sup\limits_{x\in(0,1]}f'_0(x)=f'_0(1), 
\sup\limits_{\{x<0\}}f'_0(x)=f'_0(-1)$ and $$\lim\limits_{x\to 0}\frac{f_0'(x)}{|x|^{s-1}}\mbox{ exists and it is nonzero};$$ 
\item[B3)]   $\pm 1$ are \emph{pre-periodic and repelling}: there exist $k_1, k_2, n_1, n_2$ such that
$$f_0^{n_1+k_1}(1)=f_0^{k_1}(1), \quad (f_0^{n_1})'(f_0^{k_1}(1))>1;$$
$$f_0^{n_2+k_2}(-1)=f_0^{k_2}(-1),  \quad (f_0^{n_2})'(f_0^{k_2}(-1))>1;$$
\item[B4)]  $f_0$ has \emph{negative Schwarzian derivative}: there is $\alpha<0$ such that in $I \setminus\{0\}$
$$S(f_0):=\bigg(\frac{f''_0}{f'_0}\bigg)'-\frac 12 \bigg(\frac{f''_0}{f'_0}\bigg)^2<\alpha.$$
\end{enumerate}
\begin{figure}[!ht]
\begin{center}
\includegraphics[height=5cm]{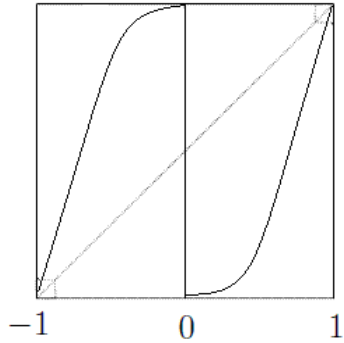}
\caption{Rovella map}\label{fig:4}
\end{center}
\end{figure}

\subsection{Rovella parameters}

The above attractor is not robust. However,  Rovella proved that it is persists in a measure theoretical sense: there exists a one parameter family of positive Lebesgue measure of $C^3$ close vector fields to the original one which have a transitive non-hyperbolic attractor.
In the proof of that result, Rovella showed that there is a set of parameters $E\subset (0,1)$  with 0 as a  full density point of $E$, i.e.
      $$\lim\limits_{a\to 0}\frac{|E\cap(0,a)|}{a}=1,$$
such that:
\begin{itemize}
\item[C1)]  there is $K_1,K_2>0$ such that for all $a\in E$ and $x\in I$ 
$$K_2|x|^{s-1}\le f'_a(x)\le K_1|x|^{s-1},$$ 
where $s=s(a)$. To simplify, we shall assume $s$ fixed, as in \cite{rovella};
\item[C2)]  there is $\lambda_c >1$ such that for all $a\in E$, the points $1$ and $-1$ have \emph{Lyapunov exponents} greater than $\lambda_c$:
\begin{equation*}
(f_a^n)'({\pm}1)>\lambda_c^n,\quad\text{for all $n\geq 0$;}
\end{equation*}

 \item [C3)] there is $\alpha>0$ such that   for all $a\in E$  the \emph{basic assumption} holds: 
 \begin{equation}
|f^{n-1}({\pm}1)|>e^{-\alpha n},\quad\text{for all $n\geq 1$;}\tag{BA}
\end{equation}
 \item[C4)]  the forward orbits of the points ${\pm}1$ under $f_a$ are dense in $[-1, 1]$ for all $a\in E$.
\end{itemize}
Metzger  used the conditions C1)-C3) in \cite{metzger}  to prove the existence of an ergodic absolutely continuous invariant probability measure for Rovella parameters. 
In order to obtain uniqueness of that measure, Metzger needed to consider a slightly smaller class of parameters (still with full density at 0), for which conditions C2) and C3) imply a strong mixing property. 

As a corollary of our main theorem, we shall deduce here the uniqueness of the ergodic absolutely continuous invariant probability measure for a smaller set 
of Rovella parameters with full density at 0 which we still denote it by $E$.
Hence, for each $a\in E$, the map $f_a$ has a \emph{Sinai-Ruelle-Bowen measure (SRB measure)} $\mu_a$:  there exists a positive Lebesgue measure set of points $x\in I$ such that
$$\lim\limits_{n\to \infty}\frac 1n \sum\limits_{i=1}^{n}{\varphi(f_a^i(x))}=\int{\varphi~d\mu_a},\quad\mbox{for every }\varphi \in C^0(I).$$
One of the goals of this work is to prove  the statistical stability of this one parameter family of maps, that is, to show that the SRB measure depends continuously (in the $L^1$ norm in the space of densities) on  the parameter $a\in E$. We shall also obtain some statistical laws for these measures.

\subsection{Statement of results}

To prove our main theorem we will use the result in   \cite{alves1},  where sufficient conditions for the statistical stability of non-uniformly expanding maps with slow recurrence to the critical set are given. Observe that the non degeneracy conditions on the critical set needed in \cite{alves1}  clearly holds in our case by the condition  C1).

\begin{definition}
We say that $f_a$ is \emph{non-uniformly expanding} if there is a $c>0$ such that for Lebesgue almost every $x\in I$

\begin{equation}\label{eq:1}
\liminf\limits_{n\to \infty}\frac 1n\sum\limits_{i=0}^{n-1}\log(f_a'(f_a^i(x)))>c.
\end{equation}
\end{definition}

\begin{definition}
We say that $f_a$ has \emph{slow recurrence to the critical set} if for every $\epsilon>0$ there exists $\delta>0$ such that for Lebesgue almost every $x\in I$

\begin{equation}\label{eq:2}
\limsup\frac 1n \sum\limits_{i=0}^{n-1}-\log \d_{\delta}(f_a^i(x),0)\leq \epsilon,
\end{equation}
where $\d_\delta$ is the $\delta$-truncated distance, defined as

$$\d_{\delta}(x,y)=\left\{
                       \begin{array}{lr}
                             |x-y|,& \mbox{if}\quad|x-y|\leq\delta,\\
                             
                             1,&     \mbox{if}\quad|x-y|>\delta.
                       \end{array}
                 \right.$$
\end{definition}

We define the expansion time function 
$$\mathcal{E}_a(x)=\min\left\{ N\geq 1 : \frac 1n\sum\limits_{i=0}^{n-1}\log f_a'(f_a^i(x))>c, \forall n\geq N\right\},$$
which is defined and finite almost every where in $I$, provided  (\ref{eq:1}) holds almost everywhere. Fixing $\epsilon>0$ and choosing $\delta>0$ conveniently, we define the recurrence time function 
$$\mathcal{R}_a(x)=\min\left\{ N\geq : \frac 1n\sum\limits_{i=0}^{n-1}-\log \d_{\delta}(f_a^i(x),0)<\epsilon,\forall n\geq N\right\},$$
which is defined and finite almost every where in $I$, as long as (\ref{eq:2}) holds almost everywhere. Now the tail set at time $n$ is the set of points which at time $n$ have not yet achieved either the uniform exponential growth of the derivative or the uniform slow recurrence:
$$\Gamma_a^n=\{x\in I: \mathcal{E}_a(x)>n\,\mbox{ or }\,\mathcal{R}_a(x)>n\}.$$

\begin{remark}
As observed in \cite[Remark 3.8]{alves1}, the slow recurrence condition is not needed    in all its strength: it is enough that \eqref{eq:2} holds for some sufficiently small $\epsilon>0$ and conveniently chosen $\delta>0$ only depending on the order $s$ and $c$. For this reason we may drop the dependence of the tail set on $\epsilon$ and $\delta$ in the notation. Moreover, the constants $c$ in \eqref{eq:1} and $\epsilon, \delta$ in \eqref{eq:2} 
can be chosen uniformly on the set of parameters $E$.
\end{remark}

For the maps considered by Rovella, first we claim that (\ref{eq:1}) holds almost everywhere and the Lebesgue measure of the set of points whose derivative has not achieved a satisfactory exponential growth at time $n$, decays exponentially fast as $n$ goes to infinity. Second, we claim that (\ref{eq:2}) also holds almost every where and the volume of the set of points that at time $n$, have been too close to the critical point, in mean, decays exponentially with $n$.

\begin{maintheorem}\label{maintheorem}
Each $f_a$, with $a\in E$, is non-uniformly expanding and has slow recurrence to the critical set.  Moreover, there are  $C>0$ and $\tau>0$ such that for all $a\in E$ and $n\in \mathbb{N}$,
$$|\Gamma^n_a|\leq Ce^{-\tau n}.$$
\end{maintheorem}

To prove this result we shall use the technique  implemented by Freitas in \cite{freitas05, freitas}. Several interesting consequences will be deduced from this main result. In particular, the uniqueness of the ergodic absolutely continuous invariant probability measure.

\begin{maincorollary}
For all $a\in E$, $f_a$ has a unique ergodic absolutely continuous invariant probability measure $\mu_a$.
\end{maincorollary}

This follows from \cite[Lemma~5.6]{bonatti}. Actually, this lemma says that for a non-uniform expanding map, each forward invariant set with positive Lebesgue measure must have full Lebesgue measure in a disk of a fixed radius (not depending on the set). Applying this to the supports of two possible ergodic absolutely continuous invariant measures, together with the existence of dense orbits given by C4), we see that there is at least a common  point in the basins of both measures. Hence, by Riesz Representation Theorem, these measures must coincide.

As an immediate consequence of Theorem~\ref{maintheorem}, \cite[Theorem~A]{alves1} and \cite[Theorem~B]{alves5} we have the strong statistical stability for the family of Rovella one-dimensional maps and the continuous variation of metric entropy.

\begin{maincorollary}
The function $E\ni a\mapsto d\mu_a/dm$ is continuous, if the $L^1$-norm is considered in the space of densities, and the entropy of $\mu_a$ varies continuously with $a\in E$.
\end{maincorollary}

Finally, we obtain several statistical properties for the SRB measures associated to the family of Rovella one-dimensional maps.
In the formulation of our statistical properties we consider the space
of \emph{H\"older continuous functions} with H\"older constant $\gamma>0$, for some $\gamma>0$. This is the space of functions $\varphi :I\to \mathbb R$ with finite H\"older norm
$$
\|\vap\|\equiv\|\vap\|_\infty+\sup_{y_1\neq y_2}\frac{|\vap(y_1)-\vap(y_2)|}{|y_1,y_2|^\gamma}.
$$
For a precise formulation of the concepts below see Appendix~\ref{ap.a}.

\begin{maincorollary}\label{co.statistical}
For all $a\in E$, the SRB measure $\mu_a$ satisfies:
\begin{enumerate}
\item exponential decay of correlations for H\"older against  $L^\infty(\mu_a)$ observables;
\item exponential large deviations for H\"older observables.
\item the Central Limit Theorem,  the vector-valued Almost Sure Invariance Principle,  
the Local Limit Theorem and the Berry-Esseen Theorem for certain H\"older observables.
\end{enumerate}
\end{maincorollary}

The exponential decay of correlations  has already been obtained in \cite{metzger}  for the subset of parameters in $E$ (still with full density at 0) for which some strong topological mixing conditions holds.

It is not difficult to see how Corollary~\ref{co.statistical} can be deduced from Theorem~\ref{maintheorem}.
Actually, it follows from Theorem~\ref{maintheorem} and \cite[Theorem~3.1]{gouezel} that each $f_a$ with $a\in E$ has a Young tower with exponential tail of recurrence times. 
Then, the exponential decay of correlations and the Central Limit Theorem follows from \cite[Theorem~4]{young}; the exponential large deviations follows from  \cite[Theorem~2.1]{MN}; 
 the vector-valued Almost Sure Invariance Principle follows from  \cite[Theorem~2.9]{MN1}; and finally, 
the Local Limit Theorem and the Berry-Esseen Theorem follow from  \cite[Theorem~1.2 \& Theorem~1.3]{G1}.

From Section~\ref{se.bounded} to Section~\ref{se.measure.tail} we shall prove Theorem~\ref{maintheorem}. 

\section{Expansion and bounded distortion}\label{se.bounded}

The following lemma gives a first property for the dynamics of maps with parameters near the parameter 0. This appears as an initial step in the construction of the set of Rovella parameters.

\begin{lemma}\label{lem1}
There are $\lambda>1$ and $\delta_0>0$ such that for any $0<\delta<\delta_0$ there are $a_0(\delta)>0$ and $c(\delta)>0$ 
such that given any $x\in I$ and $a\in [0,a_0(\delta)]$,
\begin{enumerate}
\item if $x, f_a(x),\ldots, f_a^{n-1}(x)\notin (-\delta, \delta)$, then $(f_a^n)'(x)\geq c(\delta)\lambda^n$;
\item if $x, f_a(x),\ldots, f_a^{n-1}(x)\notin (-\delta, \delta)$ and  $f_a^n(x)\in (-\delta,\delta)$, then $(f_a^n)'(x)\geq\lambda^n$;
\item if $x, f_a(x),\ldots, f_a^{n-1}(x)\notin (-\delta, \delta)$ and $f_a^n(x)\in (-e^{-1}, e^{-1})$, then $(f_a^n)'(x)\geq\frac 1e\lambda^n$.
\end{enumerate}
\end{lemma}

\begin{proof} Though not explicitly stated in the present form, this result has essentially been obtained in the proof of  \cite[Part IV, Lemma 1]{rovella}. Here we follow the main steps of that proof in order to enhance some extra properties that we need, specially the third item and the factor $c(\delta)$ in the first item.

It was proved in \cite[Lemmas  1.1 and 1.2]{rovella} that there are $\delta_0$ and $\lambda_0$ (only depending on initial vector field) such that for any $\delta<\delta_0$ there exists $a_1(\delta)$ such that for every $|y|\in(\delta, \delta_0)$ and $a\in[0,a_1(\delta)]$ there exists $p=p(a,y)$ such that
\begin{equation*}
|f^j_a(y)|>\delta_0,\quad\mbox{for}~1\leq j\leq p\quad\mbox{and}~(f_a^p)'(y)\geq \lambda_0^p.
\end{equation*}
Also, there are $m$, $\lambda_1>1$ and $a_2$ depending on $\delta_0$ such that 
\begin{equation*}
(f_a^m)'(y)\geq\lambda_1^m,\quad\mbox{whenever}\quad y, f_a(y),\ldots, f_a^{m-1}(y)\notin (-\delta_0,\delta_0),
\end{equation*}
and 
\begin{equation*}
(f_a^j)'(y)\geq\lambda_1^j,\quad\mbox{if}~f_a^j(y)\in(-\delta_0,\delta_0)\quad\mbox{for some }~1\leq j<m.
\end{equation*}
for $a\in[0,a_2]$.
Take $$c(\delta)=\min\{(K_2/\lambda_1)^r\delta^{r(s-1)}: 1\leq r\leq m\},$$ $a_0(\delta)=\min\{a_1(\delta), a_2(\delta_0)\}$ and $\lambda=\min\{\lambda_0,\lambda_1\}$. Consider $n$ and $x$ as in the statement and $a\in[0,a_0(\delta)]$. First suppose that $f_a^j(x)\notin(-\delta_0,\delta_0)$ for all $0\leq j\leq n$. We can write $n=mq+r$ for some $q\geq 0$ and $0\leq r<m$, then
\begin{equation}\label{firstcase}
(f_a^{mq+r})'(x)=(f^r_a)'(f^{qm}_a)(f^{qm}_a)'(x)\geq (K_2/\lambda_1)^r\delta_0^{r(s-1)}\lambda_1^r(\lambda_1^m)^q\geq c(\delta)\lambda^n,
\end{equation}
if $f_a^n(x)\in(-\delta_0,\delta_0)$, we replace $K_2^r\delta_0^{r(s-1)}$ with $\lambda_1^r$.

Now suppose the orbit of $x$ up to $n$ intersect $(-\delta_0,\delta_0)$. We define $0\leq t_1<\ldots<t_{\nu}<n$ as follows. Let $t_1=\min\{j\geq 0:f^j_a(x)\in(-\delta_0,\delta_0)\}$ and $t_{i+1}$ be the smallest $j$ with $t_i+p_i<j<n$ for which $f^j_a(x)\in(-\delta,\delta)$. 
For the time intervals $J=[0,t_1)$ and $J=[t_i,t_i+p_i)\cup [t_i+p_i,t_{i+1})$, $1\leq i< {\nu}$, we get
\begin{equation}\label{1.}
\prod\limits_{j\in J}f'_a(f_a^j(x))\geq \lambda^{|J|}.
\end{equation}
If $n\geq t_{\nu}+p_{\nu}$, then \eqref{1.} holds for $J=[t_{\nu},t_{\nu}+p_{\nu})$ and for $J=[t_{\nu}+p_{\nu},n)$ if $f_a^n(x)\in(-\delta,\delta)$. For the time interval $J=[t_{\nu}+p_{\nu},n)$ if $f_a^n(x)\notin(-\delta_0,\delta_0)$ and for $J=[t_{\nu}, n)$ When $n< t_{\nu}+p_{\nu}$, as in \eqref{firstcase} we have 
\begin{equation*}
\prod\limits_{j\in J}f'_a(f_a^j(x))\geq c(\delta)\lambda^{|J|}.
\end{equation*}
Without loss of generality, we may assume that the critical values $\pm1$ are fixed by $f_0$.
Given $y$ such that $f^k(y)\in (-e^{-1},e^{-1})$, let $[a,b]$ be the maximal interval of continuity of $f_0^k$ that contains $y$. 
Then
$$\frac{f_0^k(b)-f_0^k(y)}{b-y}>1-f_0^k(y)\geq 1-e^{-1}=\frac 1e(e-1),$$
$$\frac{f_0^k(y)-f_0^k(a)}{y-a}> f_0^k(y)+1\geq 1-e^{-1}=\frac 1e(e-1).$$
This implies $(f_0^k)'(y)>\frac 1e(e-1)$, because $f_0$ has negative Schawarzian derivative and otherwise it violates the minimum principle. 

Let $m$ be as above and take $\lambda_2>1$ such that $\lambda_2^m< e-1$. From what we have seen above, we have  for all small values of $a$ and  each $k<m$
$$(f_a^k)'(y)>\frac 1e\lambda_2^m,\quad\mbox{whenever }f_a^k(y)\in(-e^{-1},e^{-1}).$$
Given $n$ and $x$ as in the assumption, let $0\le k<m$ be such that $n=pm+k$ for some integer $p$. 
Taking $\lambda=\min\{\lambda_1,\lambda_2\}$, we obtain
\begin{align*}
(f^n)'(x)=(f^{pm+k})'(x)&=(f^k)'(f^{pm}(x))(f^m)'(f^{(p-1)m}(x))\ldots (f^m)'(x)\\
&\geq\frac 1e~\lambda_2^k~\lambda_1^m\ldots\lambda_1^m\\
&\geq\frac 1e~\lambda^{k+pm}\\
&=\frac 1e~\lambda^n.
\end{align*}
\end{proof}
To establish the meaning of close to critical set we introduce the neighborhoods of $0$
$$
U_m=(-e^{-m},e^{-m}),
\quad\mbox{with }m\in\mathbb{N}.
$$
Given any point $x\in I$, the orbit of $x$ will be split into \emph{free periods}, \emph{returns} and \emph{bound periods}, which will occur in this order. 

The \emph{free periods} correspond to iterates in which the orbit is not inside $U_{\Delta}$ (for some big $\Delta$) nor in a bound period. During these periods the orbit of $x$ experiences an exponential growth of its derivative, provided we take parameters close enough to the parameter value~$0$, as shown in Lemma~\ref{lem1}.

We say that $x$ has a \emph{return} at a given $j\in \mathbb N$ if $f^j(x)\in U_{\Delta}$. We shall consider two types of returns: \emph{essential} or \emph{inessential}. In order to distinguish each type we need a sequence of partitions of $I$ into intervals 
that will be defined later. The idea of this construction goes back to the work of Benedicks and Carleson in \cite{benedicks}.

The \emph{bound period}  is a period after a return time during which the orbit of $x$ is bound to the orbit of the critical point. In order to define  that precisely, suppose that the constant $\alpha$ in C3) has been taken small and let $\beta=s\alpha$. 
Let $I_m=[e^{-m-1}, e^{-m})$ for $m\geq\Delta$, $I_m=-I_{-m}$ for $m<0$, and $I_m^+=I_{m-1}\cup I_m\cup I_{m+1}$. We consider for each $I_m$, the collection of $m^2$ equal length intervals $I_{m,1}, I_{m,2},\ldots,I_{m,m^2}$, whose union is $I_m$, and order these $I_{m,i}$ as follows: if $i>j$ then $\dist(I_{m,i},0)<\dist(I_{m,j},0)$. By $I_{m,i}^+$ we denote the union of $I_{m,i}$ with the two adjacent intervals of the same type.

\begin{definition}
Given $x\in I_m^+$, let $p(x)$ be the largest integer $p$ such that for $1\le j <p$
$$|f^j(x)-f^{j-1}(-1)|
\leq e^{-\beta j},\qquad\mbox{if }m>0,$$
and
$$|f^j(x)-f^{j-1}(1)|
\leq e^{-\beta j},\qquad\mbox{if }m<0.$$
 The time interval $1,\ldots, p(x)-1$ is called the \emph{bound period} for $x$.
\end{definition}

Let us denote by $p(m)$ the smallest $p$ such that above condition holds for all $x\in I_m^+$.
The \emph{bound period }for $I_m^+$ is the time interval $1,\ldots, p(m)-1$.  The proof of the next result is a consequence of conditions (C1)-(C3) and $f'\leq 4$; see \cite[Lemma 4.2]{metzger}. Note that $f$ indicates every map $f_a$ with $a\in E$.

%
\begin{lemma}\label{lemm}
If $\Delta$ is large enough, then $p(m)$ has the following properties for each $|m|>\Delta$:
\begin{enumerate}
\item there is $C_0=C_0(\alpha, \beta)$ such that for $j=0,\ldots, p(m)-1$
      $$\frac 1{C_0}\leq \frac{(f^j)'(y)}{(f^j)'(-1)}\leq C_0,\qquad\mbox{if }y\in [-1, f(e^{-|m|+1})],$$
      $$\frac 1{C_0}\leq \frac{(f^j)'(y)}{(f^j)'(1)}\leq C_0,\qquad\mbox{if }y\in [f(-e^{-|m|+1}), 1];$$\\

\item taking $K=(-\log 4+\log (K_1/s)+s)/(\beta+\log 4)$ we have
$$\frac{s|m|}{\beta+\log 4}-K\leq p(m)\leq\frac{s+1}{\beta+\log \lambda_c}|m|;$$\\

\item 
for all $x\in I_m^+$ we have
$$(f^{p(m)})'(x)\geq e^{(1-\beta\frac{s+2}{\beta+\log\lambda_c})|m|}.$$
      \end{enumerate}
\end{lemma}

Since $(f^n)'(x)=\prod_{j=0}^{n-1}f'(f^j(x))$, the return times introduce some small factors in the derivative of the orbit of $x$, but after the bound period not only that loss on the growth of the derivative has been recovered, as we have some exponential growth again.
Note that an orbit in a bound period can enter $U_{\Delta}$ and these instants are called \emph{bound return times}. 

Now we build inductively a sequence of the partitions $\mathcal{P}_0\prec\mathcal{P}_1\prec\ldots$ of $I$ (modulo a zero Lebesgue measure set) into intervals. We also build $R_n(\omega)=\{z_1,\ldots,z_{\gamma(n)}\}$, which is the set of the return times of $\omega\in\mathcal{P}_n$ up to $n$, and  $Q_n(\omega)=\{(m_1,k_1),\ldots,(m_{\gamma(n)},k_{\gamma(n)})\}$ which records the indices of the intervals such that $f^{z_i}(\omega)\subset I_{m_i,k_i}^+$ for $1\le i \le z_{\gamma(n)}$. By construction, we shall have for all $n\in\mathbb{N}_0$
\begin{equation}\label{eq:3}
\forall \omega\in\mathcal{P}_n,\hspace{.5cm}f^{n+1}|_{\omega}\hspace{.5cm} \mbox{is a diffeomorphism}.
\end{equation}
For $n=0$ we define
$$\mathcal{P}_0=\{[-1,-\delta] , [\delta, 1]\}\cup\{I_{m,i}: |m|\geq\Delta, 1\leq i\leq m^2\}.$$
It is obvious that \eqref{eq:3} holds for every $\omega\in\mathcal{P}_0$. Set 
$R_0([-1,-\delta])=R_0([\delta, 1])=\emptyset$ and $R_0(I_{m,k})=\{0\}$, also $Q_0([-1, -\delta])=Q_0([\delta, 1])=\emptyset$ and $Q_0(I_{m,k})=\{(m,k)\}$.

Assume now that $\mathcal{P}_{n-1}$ is defined and it satisfies (\ref{eq:3}) and $R_{n-1}$, $Q_{n-1}$ are defined on each element of $\mathcal{P}_{n-1}$. Fixing $\omega\in\mathcal{P}_{n-1}$, there are three possible situations:

\begin{enumerate}

\item If $R_{n-1}(\omega)\neq\emptyset$ and $n\leq z_{\gamma(n-1)}+p(m_{\gamma(n-1)})$, we call $n$ a \emph{bound time} for $\omega$, and put $\omega\in\mathcal{P}_n$ and set $R_n(\omega)=R_{n-1}(\omega)$ and $Q_n(\omega)=Q_{n-1}(\omega)$.

\item If $R_{n-1}(\omega)=0$ or $n> z_{\gamma(n-1)}+p(m_{\gamma(n)})$ and $f^n(\omega)\cap U_{\Delta}\subset I_{\Delta,1}\cup I_{-\Delta,1}$, we call $n$ a \emph{free time} for $\omega$, put $\omega\in\mathcal{P}_n$ and set $R_n(\omega)=R_{n-1}(\omega)$ and $Q_n(\omega)=Q_{n-1}(\omega)$.

\item If the two above conditions do not hold, $\omega$ has a \emph{return situation} at time $n$. We consider two cases:

\begin{enumerate}
\item $f^n(\omega)$ does not cover completely an interval $I_{m,k}$. Since $f^n|_{\omega}$ is a diffeomorphism and $\omega$ is an interval, $f^n(\omega)$ is also an interval and thus is contained in some $I_{m,k}^+$, which is called the \emph{host interval} of the return. We call $n$ an \emph{inessential return time} for $\omega$, put $\omega\in\mathcal{P}_n$ and set    $R_n(\omega)=R_{n-1}(\omega)\cup \{n\}$, $Q_n(\omega)=Q_{n-1}(\omega)\cup \{(m,k)\}$.
\item $f^n(\omega)$ contains at least an interval $I_{m,k}$, in which case we say that $\omega$ has an \emph{essential return situation} at time $n$. Take
$$
\begin{array}{rl}
\omega_{m,k}=&f^{-n}(I_{m,k})\cap\omega,\\
\omega_+=&f^{-n}([\delta, 1])\cap\omega,\\
\omega_-=&f^{-n}([-\delta, -1])\cap\omega.
\end{array}
$$
We have $\omega\backslash f^{-n}(0)=\cup~\omega_{m,k}\cup \omega_+ \cup \omega_-$. By the induction                             hypothesis $f^n|_{\omega}$ is a diffeomorphism and then each $\omega_{m,k}$ is an interval. Moreover $f^n(\omega_{m,k})$ covers $I_{m,k}$ may except for the two end intervals. We join $\omega_{m,k}$ with its adjacent interval, if it dose not cover $I_{m,k}$ entirely. We also proceed likewise when $f^n(\omega_+)$ does not cover $I_{\Delta-1,(\Delta-1)^2}$ or $f^n(\omega_-)$ does not cover $I_{1-\Delta,(\Delta-1)^2}$. So we get a new decomposition of $\omega\backslash(\omega_+\cup\omega_-)$ into intervals $\omega_{m,k}$ (mod 0) such that $I_{m,k}\subset f^n(\omega_{m,k})\subset I_{m,k}^+$. Put $\omega_{m,k}\in\mathcal{P}_n$ for all indices $(m,k)$ such that $\omega_{m,k}\neq\emptyset$, set $R_n(\omega_{m,k})=R_{n-1}(\omega)\cup \{n\}$ and call $n$ an \emph{essential return time} for $\omega_{m,k}$. The interval $I_{m,k}$ is called the \emph{host interval} of $\omega_{m,k}$ and $Q_n(\omega_{m,k})=Q_{n-1}(\omega)\cup\{(m,k)\}$. In the case when $f^n(\omega_+)$ covers $I_{\Delta-1,(\Delta-1)^2}$ we say $n$ is an \emph{escape time} for $\omega_+$ and $R_n(\omega_+)=R_{n-1}(\omega)$, $Q_n(\omega_+)=Q_{n-1}(\omega)$. We proceed similarly for $\omega_-$. We refer to $\omega_+$ and $\omega_-$ as \emph{escaping components}.
\end{enumerate}
\end{enumerate}
To end the construction we have to verify that (\ref{eq:3}) holds for $\mathcal{P}_n$. Since for any interval $J\subset I$
$$\left.
       \begin{array}{l}
             f^n|_J ~~\mbox{is a diffeomorphism}\\
             0\notin f^n(J)
       \end{array}
\right\}
\Rightarrow f^{n+1}|_J ~~\mbox{is a diffeomorphism,}
$$
we are left to prove $0\notin f^n(\omega)$ for all $\omega\in\mathcal{P}_n$. Take $\omega\in\mathcal{P}_n$. If $n$ is a free time, there is nothing to prove. If $n$ is a return time, either essential or inessential, we have by construction $f^n(\omega)\subset I_{m,k}^+$ for some $|m|\geq\Delta$ and $k=1, 2,\ldots, m^2$, and thus $0\notin f^n(\omega)$. If $n$ is a bound time, then by definition of bound period and the basic assumption C3) for all $x\in\omega$ whit $m_{\gamma(n-1)}>0$ we have
$$\begin{array}{rl}
|f^n(x)|\geq&|f^{n-1-z_{\gamma(n-1)}}(-1)|-|f^n(x)-f^{n-1-z_{\gamma(n-1)}}(-1)|\\
\geq&|f^{n-1-z_{\gamma(n-1)}}(-1)|-|f^{n-z_{\gamma(n-1)}}(f^{z_{\gamma(n-1)}}(x))-f^{n-1-z_{\gamma(n-1)}}(-1)|\\
\geq&e^{-\alpha(n-z_{\gamma(n-1)})}-e^{-\beta(n-z_{\gamma(n-1)})}\\
\geq&e^{-\alpha(n-z_{\gamma(n-1)})}(1-e^{-(\beta-\alpha)(n-z_{\gamma(n-1)})})\\
>&0,\hspace{2cm}\mbox{since}~~\beta-\alpha>0~~\mbox{and}~~n-z_{\gamma(n-1)}\geq 1.
\end{array}$$
The same conclusion can be drawn for $\omega\in\mathcal{P}_n$ with $m_{\gamma(n-1)}<0$. We just need to replace $-1$ by $1$ in the above calculation.

Now we want to estimate the length of $|f^n(\omega)|$. The next lemma follows from Lemmas~\ref{lem1} and \ref{lemm} above exactly in the same way as  in \cite[Lemma~4.1]{freitas05}.
\begin{lemma}\label{lem3}
Let $z$ be a return time for $\omega\in\mathcal{P}_{n-1}$ with host interval $I_{m,k}^+$ and  $p=p(m)$;
\begin{enumerate}
\item assuming $z^*\leq n-1$ is the next return time for $\omega$, either essential or inessential, and defining $q=z^*-(z+p)$, we have $$|f^{z^*}(\omega)|\geq\lambda^qe^{\left(1-\beta\frac{s+2}{\beta+\log \lambda_c}\right)|m|}|f^z(\omega)|,$$
and for sufficiently large $\Delta$ it follows that $|f^{z^*}(\omega)|\geq 2|f^z(\omega)|$;
\item if $z$ is the last return time of $\omega$ up to $n-1$ and $n$ is either a free time or a return situation for $\omega$ and put $q=n-(z+p)$, then 
\begin{enumerate}
\item $|f^n(\omega)|\geq c(\delta)\lambda^q e^{\left(1-\beta\frac{s+2}{\beta+\log \lambda_c}\right)|m|}|f^z(\omega)|$,
\item if $z$ is an essential return and for a sufficiently large $\Delta$
$$|f^n(\omega)|\geq c(\delta)\lambda^q e^{-\beta\frac{s+3}{\beta+\log \lambda_c}|m|};$$
\end{enumerate}
\item if $z$ is the last return time of $\omega$ up to $n-1$, $n$ is a return situation and $f^n(\omega)\subset U_{1}$, then for a sufficiently large $\Delta$ and $q=n-(z+p)$
\begin{enumerate}
\item $|f^n(\omega)|\geq \lambda^q e^{\left(1-\beta\frac{s+3}{\beta+\log \lambda_c}\right)|m|}|f^z(\omega)|\geq 2|f^n(\omega)|$,
\item $|f^n(\omega)|\geq \lambda^q e^{-\beta\frac{s+3}{\beta+\log \lambda_c}|m|}$, if $z$ is an essential return.
\end{enumerate}
\end{enumerate}
\end{lemma}

The next lemma tells us that the escape component returns considerably large in the return situation after the escape time. Though the content is similar to a lemma in \cite[Lemma~4.2]{Fr06}, here we cannot use the quadratic expression of the map.

\begin{lemma}\label{lemi}
If $\omega\in\mathcal{P}_{t}$ is an escape component, then in the next return situation $t_1$ for $\omega$ we have
$$|f^{t_1}(\omega)|\geq e^{-\beta\frac{s+3}{\beta+\log\lambda_c}\Delta}.$$
\end{lemma}
\begin{proof}
If $f^{t_1}(\omega)\not\subseteq U_{1}$, there is nothing to prove. So suppose that $f^{t_1}(\omega)\subseteq U_{1}$. Since $\omega$ is an escape component at time $t$ it follows that
$$
f^{t}(\omega)\supseteq I_{m,m^2},\quad\mbox{with}~|m|=\Delta-1.
$$
First we assume that $f^t(\omega)\subseteq I_m$. Without loss of generality, assume that $m>0$. By definition of $p=p(\Delta)$, for every $j=1,\ldots, p-1$ and $x\in\omega$
$$|f^{j-1}(-1)|-|f^{t+j}(x)|\leq |f^{j+t}(x)-f^{j-1}(-1)|\leq e^{-\beta j}.$$
Hence
\begin{align*}
|f^{j+t}(x)|\geq& |f^{j-1}(-1)|-e^{-\beta j}\\
\geq& e^{-\alpha j}-e^{-\beta j}=e^{-\alpha j}(1-e^{(\alpha-\beta)j}),\hspace{.5cm}(BA)\\
\geq& e^{-\alpha p}(1-e^{\alpha-\beta}),\hspace{3.5cm}j<p,~\alpha-\beta<0\\
\geq& e^{-\alpha\frac{s+1}{\beta+\log\lambda_c}\Delta}(1-e^{\alpha-\beta}),\hspace{2.3cm}p\leq\frac{s+1}{\beta+\log \lambda_c}\Delta\\
\geq& e^{-\alpha\frac{s+2}{\beta+\log\lambda_c}\Delta},\hspace{4.1cm}1-e^{\alpha-\beta}>e^{-\alpha\frac{1}{\beta+\log\lambda_c}\Delta}\\
\geq& e^{-\Delta}\hspace{5.6cm}\alpha\frac{s+2}{\beta+\log\lambda_c}<1.
\end{align*}
Therefore, there is no return situation during times $j=1,\ldots, p-1$ and then 
$t_1-t\geq p$. Hence
\begin{align*}
|f^{t_1}(\omega)|=&|f^{t_1-t}(f^t(\omega))|=(f^{t_1-t})'(f^t(x))|f^t(\omega)|,\quad\mbox{for some }x\in\omega\\
=&(f^q)'(f^{p+t}(x))(f^{p})'(f^t(x))|f^t(\omega)|,\quad q=t_1-t-p\\
\geq&\frac 1e \lambda^q e^{(1-\beta\frac{s+2}{\beta+\log\lambda_c})\Delta}|f^t(\omega)|,\quad\mbox{part 3 of Lemmas }\ref{lem1}, \ref{lemm}\\
\geq&\frac 1{e\Delta^2}e^{\frac{\beta}{\beta+\log\lambda_c}\Delta}\lambda^q e^{(1-\beta\frac{s+3}{\beta+\log\lambda_c})\Delta}e^{-\Delta},\quad|f^t(\omega)|\geq\frac{e^{-\Delta+1}-e^{-\Delta}}{(\Delta-1)^2}>\frac{e^{-\Delta}}{\Delta^2}\\
\geq&e^{-\beta\frac{s+3}{\beta+\log\lambda_c}\Delta},\quad\Delta\mbox{ large enough}. 
\end{align*}
Suppose now that $I_m\subseteq f^t(\omega)$. If $t_1-t\geq p(\Delta)$, then we obtain the result by knowing that 
$$|f^{t_1}(\omega)|=|f^{t_1-t}(f^t(\omega))|\geq|f^{t_1-t}(I_m)|\geq (f^{t_1-t})'(x)|I_m|,\qquad\mbox{for some }x\in I_m,$$
and following the proof in the previous case. Now we consider the case when $t_1-t< p(\Delta)$. In this case, the points in $I_m$ are in the bound period. So, for every $x\in I_m$ 
$$|f^{t_1-t}(x)|\geq e^{-\alpha(t_1-t)}-e^{-\beta(t_1-t)}\geq e^{-\beta(t_1-t)},$$
second inequality holds because $\beta>\alpha$ 
and $t_1-t$ is large by taking $\Delta$ to be big enough. Therefore 
\begin{align*}
|f^{t_1}(\omega)|&\geq \dist(f^{t-1-t}(I_m), U_{\Delta}) 
\\
&\geq e^{-\beta(t_1-t)}-e^{-\Delta} 
\\
&\geq e^{-\beta p(\Delta)}-e^{-\Delta} 
\\
&\geq e^{-\beta \frac{s+2}{\beta+\log\lambda_c}\Delta}-e^{-\Delta} 
\\
&\geq e^{-\beta \frac{s+3}{\beta+\log\lambda_c}\Delta}.
\end{align*}

\end{proof}

Before we give the bounded distortion result, we prove a preliminary result. Let us define $\mathfrak{D}(\omega)$ as a distance of the interval $\omega$ from critical point, i.e.
$$\mathfrak{D}(\omega)=\inf\limits_{x\in\omega}|x|.$$ 

\begin{lemma}\label{lem2}
Given an interval $\omega$ such that $|\omega|\leq \mathfrak{D}(\omega)$, there exists $C_1>0$ such that
$$\sup\limits_{x, y\in\omega}\left|\frac{f''(x)}{f'(y)}\right|\leq \frac{C_1}{\mathfrak{D(\omega)}}.$$
\end{lemma}

\begin{proof}
Outside some small neighborhood $(-\epsilon_0,\epsilon_0)$ of the origin, both $f''$, $f'$ are bounded from above and below by constants which depend only on the map and $\epsilon_0$. Then the 
result follows immediately if $\omega\cap (-\epsilon_0,\epsilon_0)=\emptyset$. Inside a neighborhood $(-2\epsilon_0,2\epsilon_0)$ we have 
$$K_2|x|^{s-1}\leq f'(x)\leq K_1|x|^{s-1},$$
$$K'_2|x|^{s-2}\leq |f''(x)|\leq K'_1|x|^{s-2}.$$
Now suppose that $\omega\cap (-\epsilon_0,\epsilon_0)\neq\emptyset$. Since $|\omega|\leq\mathfrak{D}(\omega)<\epsilon_0$, then we have $\omega\subset (-2\epsilon_0,2\epsilon_0)$ and
$$|f''(x)|\leq K'_1 \mathfrak{D}^{s-2}(\omega),\hspace{1 cm}\mbox{for~all}~x\in\omega;$$
$$|f'(y)|\geq K_2 \mathfrak{D}^{s-1}(\omega),\hspace{1 cm}\mbox{for~all}~y\in\omega.$$
Hence, by taking quotient and supremum, the result follows.
\end{proof}

\begin{proposition}[Bounded Distortion]\label{distortion}
There is  $C_2>0$ (independent of $\Delta$) such that for all $\omega\in\mathcal{P}_n$ with $f^{n+1}(\omega)\subset U_{\Delta_0}$, where $\Delta_0=\left\lceil{\beta\frac {s+2}{\beta+\log\lambda_c}\Delta}\right\rceil$, and all $x, y\in \omega$
$$\frac{(f^{n+1})'(x)}{(f^{n+1})'(y)}\leq C_2.$$
\end{proposition}
\begin{proof}
Consider the set of return times and host indices of $\omega$ as $R_{n+1}(\omega)=\{z_0,\ldots, z_{\gamma(n+1)}\}$ and $Q_{n+1}(\omega)=\{(m_1,k_1),\ldots, (m_{\gamma(n+1)},k_{\gamma(n+1)})\}$ respectively. Let $\omega_j=f^j(\omega)$, $p_j=p(\omega_{z_j})$, $x_j=f^j(x)$ and $y_j=f^j(y)$.
Using the chain rule, we write
$$
\frac{(f^{n+1})'(x)}{(f^{n+1})'(y)}=\prod\limits_{j=0}^{n}\frac{f'(x_j)}{f'(y_j)}\leq \prod\limits_{j=0}^{n}\left(1+\left|\frac{f'(x_j)-f'(y_j)}{f'(y_j)}\right|\right).
$$
According to this, the proof is completed by showing that
$$
S=\sum\limits_{j=0}^{n}\left|\frac{f'(x_j)-f'(y_j)}{f'(y_j)}\right|=\sum\limits_{j=0}^{n}|A_j|
$$
is uniformly bounded. By the Mean Value Theorem we also have
$|f'(x_j)-f'(y_j)|\leq|f''(\xi)||x_j-y_j|$ for some $\xi$ between $x_j$ and $y_j$. Therefore, 
$$
|A_j|\leq\sup\limits_{\xi_1, \xi_2\in\omega_j}\left|\frac{f''(\xi_1)}{f'(\xi_2)}\right||x_j-y_j|.
$$

We first estimate the contribution of the free period between $z_{i-1}$ and $z_i$ for the sum $S$:
Without loss of generality, assume that $m_i>0$. Since $\omega_{z_i}\subset I_{m_i}^+\subset [e^{-|m_i|-2}, e^{-|m_i|+1}),$
\begin{equation}\label{1.1}
|\omega_{z_i}|\leq 3\frac{e^{-|m_i|+1}-e^{-|m_i|}}{(|m_i|-1)^2}=3\frac{e-1}{(|m_i|-1)^2}e^{-|m_i|}.
\end{equation}
For $j\in[z_{i-1}+p_{i-1}, z_i-1]$,

\begin{align*}
|\omega_{z_i}|\geq&|f^{z_i-j}(x_j)-f^{z_i-j}(y_j)|\\
=&|(f^{z_i-j})'(\xi)||x_j-y_j|,\qquad\mbox{for some }\xi\mbox{ between }x_j\mbox{ and }y_j\\
\geq&\lambda^{z_i-j}|x_j-y_j|,\qquad\mbox{by Lemma }\ref{lem1}.
\end{align*}
Moreover, $\omega_j$ stays out of $(-\delta, \delta)$, which implies $\delta<\mathfrak{D}(\omega_j)$. Then the hypothesis of Lemma~\ref{lem2} holds:
$$
|\omega_j|\leq\lambda^{j-z_i}|\omega_{z_i}|\leq3\lambda^{j-z_i}\frac{e-1}{(|m_i|-1)^2}e^{-|m_i|}\ll\delta<\mathfrak{D}(\omega_j).
$$
This gives
$$
\sum\limits_{j=z_{i-1}+p_{i-1}}^{z_i-1}|A_j|=\sum\limits_{j=z_{i-1}+p_{i-1}}^{z_i-1}\sup\limits_{\xi_1, \xi_2\in\omega_j}\left|\frac{f''(\xi_1)}{f'(\xi_2)}\right||x_j-y_j|
\leq\frac{ C_1|\omega_{z_i}|}{\mathfrak{D}(\omega_j)}\sum\limits_{j=z_{i-1}+p_{i-1}}^{z_i-1}\lambda^{j-z_i}\leq \frac{C_1}{\lambda-1}|\omega_{z_i}|e^{|m_i|}.
$$

We have seen in \eqref{1.1} that the hypothesis of Lemma \ref{lem2} for $\omega_{z_i}$ is also satisfied ($e^{-|m_i|-2}\leq\mathfrak{D}(\omega_{z_i})$).
Then the contribution of the return time $z_i$ is:
$$
|A_{z_i}|\leq\sup\limits_{\xi_1, \xi_2\in\omega_j}\left|\frac{f''(\xi_1)}{f'(\xi_2)}\right||\omega_{z_i}|\leq \frac{C_1}{\mathfrak{D}(\omega_{z_i})}|\omega_{z_i}|\leq C_1|\omega_{z_i}|e^{|m_i|+2}=e^2C_1|\omega_{z_i}|e^{|m_i|}.
$$

Now we compute the contribution of the bound period, $j\in(z_i,z_i+p_i)$. Let's take $\omega^*=(0,e^{-|m_i|+1}]$ and $\omega_j^*=f^j(\omega^*)$. As we are in the bound period, $|\omega_j^*|\leq e^{-\beta j}$ and 
\begin{align*}
|\omega_j^*|=&|(f^{j-1})'(\xi)||\omega_1^*|\qquad\mbox{for some }\xi\in\omega^*_1\subset[-1, f(e^{-|m_i|+1})]\\
\geq& \frac 1{C_0}|(f^{j-1})'(-1)||\omega_1^*|\qquad\mbox{Lemma} \ref{lemm}.
\end{align*}
Therefore, $|(f^{j-1})'(-1)|\leq C_0\frac{|\omega_j^*|}{|\omega_1^*|}\leq C_0\frac{e^{-\beta j}}{|\omega_1^*|}$. Now

\begin{align*}
|f^j(\omega_{z_i})|=&|\omega_{z_i}||(f^j)'(\xi)|,\qquad\mbox{for some }\xi\in\omega_{z_i}\\
=&|\omega_{z_i}||(f^{j-1})'(f(\xi))||f'(\xi)|\\
\leq& C_0|\omega_{z_i}||f'(\xi)||(f^{j-1})'(-1)|\\
\leq& C_0^2 |\omega_{z_i}|K_1e^{(1-|m_i|)(s-1)}\frac{e^{-\beta j}}{|\omega_1^*|},
\end{align*}
but $\xi\in\omega_{z_i}\subset [e^{-|m_i|-2}, e^{-|m_i|+1})$, which implies that $|f'(\xi)|\leq K_1 e^{(1-|m_i|)(s-1)}$.
And 
$$|\omega_1^*|=f\left(e^{-|m_i|+1}\right)+1\geq \frac{K_2}{s}e^{(1-|m_i|)s}.$$
Consequently,
\begin{equation}\label{1.3}
|f^j(\omega_{z_i})|\leq C_0^2K_1\frac s{K_2}|\omega_{z_i}|e^{(1-|m_i|)(s-1)}e^{-\beta j}e^{(|m_i|-1)s}=C'|\omega_{z_i}|e^{-\beta j}e^{|m_i|}.
\end{equation}
Combining \eqref{1.1} and \eqref{1.3},
\begin{equation*}
|f^j(\omega_{z_i})|\leq 3C'\frac{e^2(e-1)}{(|m_i|-1)^2}e^{-|m_i|-2}e^{-\beta j}e^{|m_i|}\leq 3C'\frac{e-1}{(|m_i|-1)^2}e^{-\beta j}.
\end{equation*}
If $\Delta$ is large enough such that  $3C'\frac{e-1}{(\Delta-1)^2}+1\leq e^{\beta-\alpha}$, then
\begin{equation}\label{444}
3C'\frac{e-1}{(|m_i|-1)^2}+1\leq 3C'\frac{e-1}{(\Delta-1)^2}+1\leq e^{\beta-\alpha}\leq e^{(\beta-\alpha)j}. 
\end{equation}
On the other hand, for $j$ in the bound period and $x\in\omega_{z_i}$,
$$
e^{-\beta j}\geq |f^{j-1}(-1)-f^{j}(x)|\geq |f^{j-1}(-1)|-|f^{j}(x)|\geq e^{-\alpha j}-|f^{j}(x)|,
$$
which implies that 
$\mathfrak{D}(f^j(\omega_{z_i}))\geq e^{-\alpha j}- e^{-\beta j}$. Thus by \eqref{444}, we have
$$\mathfrak{D}(f^j(\omega_{z_i}))\geq e^{-\alpha j}- e^{-\beta j}\geq 3C'\frac{e-1}{(|m_i|-1)^2}e^{-\beta j}\geq |f^j(\omega_{z_i})|.$$
Using Lemma \ref{lem2},
$$\sup\limits_{\xi_1, \xi_2\in f^j(\omega_{z_i})}\left|\frac{f''(\xi_1)}{f'(\xi_2)}\right|\leq\frac{C_1}{\mathfrak{D}(f^j(\omega_{z_i}))}\leq\frac{C_1}{e^{-\alpha j}-e^{-\beta j}}.$$
Now we are ready to compute the contribution of the bound period
\begin{align*}
\sum\limits_{j=z_i+1}^{z_i+p_i-1}|A_j|
\leq&\sum\limits_{j=1}^{p_i-1}|f^j(\omega_{z_i})|\sup\limits_{\xi_1, \xi_2\in f^j(\omega_{z_i})}\left|\frac{f''(\xi_1)}{f'(\xi_2)}\right|\\
\leq& \sum\limits_{j=1}^{p_i-1} C_1C'|\omega_{z_i}|e^{|m_i|}\frac{e^{-\beta j}}{e^{-\alpha j}-e^{-\beta j}}\\
\leq& C_1C'\frac{e^{\alpha-\beta}}{(1-e^{\alpha-\beta})^2}|\omega_{z_i}|e^{|m_i|}.
\end{align*}
If we assume that $n<z_{\gamma}+p_{\gamma}$, then we split the sum $\sum\limits_{j=0}^{n}|A_j|$ into three sums according to free period, return time and bound period: 
\begin{align*}
\sum\limits_{j=0}^{n}|A_j|=&\sum\limits_{i=1}^{\gamma(n+1)}\left(\sum\limits_{j=z_{i-1}+p_{i-1}}^{z_{i}-1}|A_j|+|A_{z_{i}}|+\sum\limits_{j=z_{i}+1}^{z_{i}+p_{i}-1}|A_j|\right)\\
\leq&\sum\limits_{i=1}^{\gamma(n+1)}\left(\frac{C_1}{\lambda-1}e^{|m_{i}|}|\omega_{z_i}|+e^2C_1|\omega_{z_i}|e^{|m_{i}|}+C_1C'\frac{e^{\alpha-\beta}}{(1-e^{\alpha-\beta})^2}|\omega_{z_i}|e^{|m_{i}|}\right)\\
\leq& C''\sum\limits_{i=1}^{\gamma(n+1)}|\omega_{z_i}|e^{|m_{i}|}\\
\leq& C''\sum\limits_{R\geq\Delta}e^{R}\sum\limits_{i:|m_{i}|=R}|\omega_{z_i}|\\
\leq& 2\frac{e-1}{e}C''\sum\limits_{R\geq\Delta}e^{R}\frac{e^{-R}}{R^2}
\end{align*}
We have last inequality, because by the first part of Lemma~\ref{lem3}, if $\{z_{i_j}: j=1,\ldots, r\}$ is a set of returns with depth $R$ that is in an increasing order, then $$\sum\limits_{i:|m_{i}|=R}|\omega_{z_i}|=\sum\limits_{j=1}^{r}|\omega_{z_{i_j}}|\leq\sum\limits_{j=1}^r2^{-r+j}|\omega_{z_{i_r}}|\leq 2|\omega_{z_{i_r}}|\leq2\frac{e^{-R}-e^{-R-1}}{R^2}.$$

Finally, if $n+1\geq z_{\gamma}+p_{\gamma}$, then we take care of the last piece of free period, i.e. $j\in[z_{\gamma}+p_{\gamma}, n]$. In this period of time we have 
\begin{align*}
|\omega_{n+1}|&=(f^{n+1-j})'(\xi)|\omega_j|\quad\mbox{for some }\xi\in\omega_j\\
&\geq\frac 1e\lambda^{n+1-j}|\omega_j|.
\end{align*}
First suppose that $|\omega_{n+1}|\leq \frac \lambda e\delta$. Then $|\omega_j|\leq \lambda^{j-n}\delta\leq \delta$. Since during the last free period $\omega_j$ is outside of $(-\delta,\delta)$, we have $\mathfrak{D}(\omega_j)\geq\delta$. Therefore $|\omega_j|\leq\mathfrak{D}(\omega_j)$.  Lemma \ref{lem2} gives that $$\sup\limits_{\xi_1,\xi_2}\left|\frac{f''(\xi_1)}{f'(\xi_2)}\right|\leq\frac{C_1}{\mathfrak{D}(\omega_j)},$$ and so
$$
\sum\limits_{j=z_{\gamma}+p_{\gamma}}^{n}\sup\limits_{\xi_1,\xi_2\in\omega_j}\left|\frac{f''(\xi_1)}{f'(\xi_2)}\right||x_j-y_j|\leq\sum\limits_{j=z_{\gamma}+p_{\gamma}}^{n}\frac{C_1}{\mathfrak{D}(\omega_j)}\delta\lambda^{j-n}\leq\frac{\lambda C_1}{\lambda-1}.
$$
Now suppose that $|\omega_{n+1}|>\frac\lambda e\delta$. Let $q\geq z_{\gamma}+p_{\gamma}$ be the last integer such that $|\omega_q|\leq\frac \lambda e\delta$. Then
$$
\frac 1e\geq|\omega_{n+1}|\geq\frac 1e\lambda^{n-q}|\omega_{q+1}|>\frac {1}{e^2}\lambda^{n+1-q}\delta,
$$
which implies that $n-q<\frac{1-\log\delta}{\log\lambda}-1$. On the other hand, for parameter value $a$ sufficiently close to 0 we have
$$
\sup\limits_{\xi\in I\setminus\{0\}}|f_a^j(\xi)-f_0^j(\xi)|\leq e^{-\Delta_0}\quad\mbox{for all }j=0, 1, \ldots, n-q.
$$
So, $f_0^{n-q}(\omega_q)$ is contained in a small neighborhood of 0 and since $f_0$ is a Misiurewicz map with negative schwartzian derivative, there exists a constant $K<\infty$ independent of $\Delta$ such that 
$
\frac{(f_0^{n-q})'(x_q)}{(f_0^{n-q})'(y_q)}\leq K
$; see
\cite[Proposition V.6.1]{demelo}. By knowing that $n-q$ is bounded and derivatives of $f_a$ depend continuously on $a$, we may take $a$ sufficiently close to $0$ in order to have 
$$
\frac{(f_a^{n-q})'(x_q)}{(f_a^{n-q})'(y_q)}\leq 2K.
$$
\end{proof}

\section{Return depths}\label{se.retrndeoths}
In this section we look more closely at return depths. This provides the first basic idea for the proof of our main theorem: 
the total sum of the depth of bound and inessential returns is proportional to the depth of the essential return preceding them. 
Though we follow the same strategy of \cite{Fr06}, we include detailed proofs for the sake of completeness, as the conditions in our maps and some estimates are different.

As it was seen, there are three types or returns:  essential, bound and inessential, which are denoted by $t$, $u$ and $v$ respectively. Each essential return 
might be followed by some bound returns 
and inessential return. 
We proceed to show that the depth of an inessential return is not greater that the depth of an inessential return that precedes it.  
\begin{lemma}\label{lem4}
Suppose $t$ is an essential return time for $\omega\in\mathcal{P}_{t}$ with $I_{m,k}\subset f^{t}(\omega)\subset I^+_{m,k}$. The depth of each inessential return $v$ before the next essential return is not grater than $\eta$.
\end{lemma}

\begin{proof}
It follows from the first item of  Lemma \ref{lem3}  that
$$
|f^v(\omega)|\geq 2|f^t(\omega)|\geq 2|I_{m,k}|.
$$
But since $v$ is an inessential return time, $f^v(\omega)\subseteq I_{m_1,k_1}$ for some $m_1\geq\Delta$ and $1\leq k_1\leq m_1^2$.
Therefor, $|I_{m,k}|\leq |I_{m_1,k_1}|$, which implies that $m>m_1$.
\end{proof}
The same conclusion can be drawn for bound returns.
\begin{lemma}\label{lem5}
Suppose $t$ is an essential or an inessential return time for $\omega$ with $f^t(\omega)\subset I^+_{m,k}$ and $p$ is the bound period associated to this return. Then for $x\in\omega$, if the orbit of $x$ returns to $U_{\Delta}$ between $t$ and $t+p$ the depth is not grater than $m$.
\end{lemma}

\begin{proof}
There is no loss of generality in assuming $m>0$. Since we are in the bound period
$$
|f^{j-1}(-1)|-|f^{t+j}(x)|\leq |f^{j-1}(-1)-f^{t+j}(x)|
\leq e^{-\beta j},\qquad\mbox{for }j=1,...,p-1.
$$
Accordingly
\begin{align*}
|f^{t+j}(x)|\geq& |f^{j-1}(-1)|-e^{-\beta j}\\
\geq& e^{-\alpha j}-e^{-\beta j}=e^{-\alpha j}(1-e^{(\alpha-\beta)j}),\hspace{.5cm}(BA)\\
\geq& e^{-\alpha p}(1-e^{\alpha-\beta}),\hspace{3.5cm}j<p,~\alpha-\beta<0\\
\geq& e^{-\alpha\frac{s+1}{\beta+\log\lambda_c}|m|}(1-e^{\alpha-\beta}),\hspace{2.3cm}p\leq\frac{s+1}{\beta+\log \lambda_c}|m|\\
\geq& e^{-\alpha\frac{s+2}{\beta+\log\lambda_c}|m|},\hspace{4.1cm}1-e^{\alpha-\beta}>e^{-\alpha\frac{1}{\beta+\log\lambda_c}|m|}\\
\geq& e^{-|m|}.
\end{align*}
The last inequality holds by taking $\alpha$ small such that $\alpha\frac{s+2}{\beta+\log\lambda_c}<1$.
\end{proof}

In the proof of the next lemma we shall use the \emph{free period assumption} (FA) in  \cite[page~255]{rovella}: parameters $a\in E$ have been chosen in such a way that if
$$H^{\pm}_n=\#\{i\leq n: i\mbox{ is a free time for }\pm 1\}$$ 
then
$$H^{\pm}_n\geq (1-\epsilon)n,\qquad\mbox{for all }n\geq 1,$$
for some small positive constant $\epsilon$ (not depending on $a$).

\begin{lemma}\label{1lem}
There is  $C_3>0$ such that if $t$ is a return time for $\omega\in\mathcal{P}_t$ with  $I^+_{m,k}$ the host interval,  $p$ is the bound period associated with this return, and $S$ is the sum of the depth of all bound returns between $t$ and $t+p$ plus the depth of the return $t$ that originated the bound period $p$, then $S\leq C_3|m|$.
\end{lemma}

\begin{proof}
Suppose $u_1$ is the first time between $t$ and $t+p$ that the orbit of $x\in\omega$ enters $U_{\Delta}$. Since the bound period at time $t$ is not finished yet we say at time $u_1$ there is just one active binding to the critical point and we call $u_1$ is a bound return of level 1. At time $u_1$ the orbit of $x$ establishes a new binding to the critical point which ends before $t+p$ that we denote by $p_1$.  
During the period from $u_1$ to $u_1+p_1$ a new return may happen and its level is at least 2 because there are at least two active bindings: the one initiated at $t$ and the one initiated at $u_1$. But new bound returns of level 1 may occur after $u_1 + p_1$. In this way we define the notion of bound return of level $i$ at which the orbit has already initiated exactly $i$ bindings to the critical point and all of them are still active. By active we mean that the respective bound periods have not finished yet. 

Here we use free period assumption 
%
which gives that from $t$ to $t + p$, the orbit of $x\in\omega$ can spend at most the fraction of time $\epsilon p$ in bound periods. Now suppose $n$ denotes the number of bound returns of level 1 at $u_1,\ldots, u_n$ with depths $m_1,\ldots, m_n$ and bound periods $p_1,\ldots, p_n$. Then Lemma \ref{lemm} applies:
$$\frac{s-1}{\beta+\log 4}\sum\limits_{i=1}^{n}|m_i|\leq\sum\limits_{i=1}^{n} p_i\leq \epsilon p\leq \epsilon\frac{s+1}{\beta+\log\lambda_c}|m|.$$
Hence
$$\sum\limits_{i=1}^{n} |m_i|\leq \epsilon \frac{s+1}{\beta+\log\lambda_c}\frac{\beta+\log 4}{s-1}|m|.$$
Now let $n_i$ denote the number of bound returns of level 2 within the $i$-th bound period of level 1 at $u_{i1},\ldots, u_{in_i}$ with depths $m_{i1},\ldots, m_{in_i}$ and bound periods $p_{i1},\ldots, p_{in_i}$. Then
$$\frac{s-1}{\beta+\log 4}\sum\limits_{j=1}^{n_i}|m_{ij}|\leq\sum\limits_{j=1}^{n_i} p_{ij}\leq \epsilon p_i\leq \epsilon\frac{s+1}{\beta+\log\lambda_c}|m_i|.$$
Consequently
$$\sum\limits_{i=1}^{n}\sum\limits_{j=1}^{n_i}|m_{i,j}|\leq \sum\limits_{i}^{n}\epsilon \frac{s+1}{\beta+\log\lambda_c}\frac{\beta+\log 4}{s-1}|m_i|\leq (\epsilon \frac{s+1}{\beta+\log\lambda_c}\frac{\beta+\log 4}{s-1})^2|m|.$$
By induction we have
$$S\leq \sum\limits_{i=0}^{\infty}\left(\alpha \frac{s+1}{\beta+\log\lambda_c}\frac{\beta+\log 4}{s-1}\right)^i|m|\leq C_3|m|,$$
for $$\alpha \frac{s+1}{\beta+\log\lambda_c}\frac{\beta+\log 4}{s-1}<1\quad\mbox{and}\quad C_3=\left({1-\alpha \frac{s+1}{\beta+\log\lambda_c}\frac{\beta+\log 4}{s-1}}\right)^{-1}.$$
\end{proof}

\begin{lemma}\label{2lem}
There is $C_4>0$ such that if  $t$ is an essential return time for $\omega\in\mathcal{P}_t$ with host interval $I^+_{m,k}$ and
$S$ is the sum of the depths of all free inessential returns before the next essential return, then $S\leq C_4|m|$.
\end{lemma}

\begin{proof}
Suppose that $n$ is the number of inessential returns before the next essential return situation with time occurrence $v_1, \ldots, v_n$, depths $m_1,\ldots, m_n$ and bound periods $p_1,\ldots, p_n$. Also denote by $v_{n+1}$ the next essential return situation. Let $\omega_j= f^{v_j}(\omega)$ for $j=1,\ldots,n+1$. We get by Lemma~\ref{lem3}
$$
|\omega_1|\geq
\lambda^{q_0}e^{-\beta\frac{s+3}{\beta+\log\lambda_c}|m|}\quad\mbox{and}\quad\frac{|\omega_{j+1}|}{|\omega_j|}\geq \lambda^{q_j}e^{(1-\beta\frac{s+2}{\beta+\log\lambda_c})|m_j|},
$$
where $q_j=v_{j+1}-(v_j+p_j)$ for $j=1,...,n$. Using equality 
$$
|\omega_{n+1}|=|\omega_1|\prod\limits_{j=1}^{n}\frac{|\omega_{j+1}|}{|\omega_j|}
$$
we have 
$$
\lambda^{\sum\limits_{j=0}^{n}q_j}e^{-\beta\frac{s+3}{\beta+\log\lambda_c}|m|}e^{\sum\limits_{j=1}^{n}(1-\beta\frac{s+2}{\beta+\log\lambda_c})|m_j|}\leq |\omega_{n+1}|\leq e.
$$
Therefore
$$
\sum\limits_{j=1}^{n}\left(1-\beta\frac{s+2}{\beta+\log\lambda_c}\right)|m_j|\leq \beta\frac{s+3}{\beta+\log\lambda_c}|m|+1-\log\lambda\sum\limits_{j=0}^{n}q_j\leq \beta\frac{s+3}{\beta+\log\lambda_c}|m|+1,
$$
which easily gives the desired conclusion.
\end{proof}

\section{Probability of essential returns with a certain depth}
  
In the previous section we studied the depth of returns and we saw that only essential returns matter. Now we proceed with the study of second basic idea for the proof of our main theorem: 
the chance of occurring very deep essential returns is very small. 
The main ingredient of the proof is bounded distortion.  Again we follow the same strategy of \cite{Fr06}.  

For each $x\in I$ and $n\in \mathbb{N}$ there is an unique $\omega\in\mathcal{P}_n$ such that $x\in\omega$. Now let $r_n(x)$ be the number of essential return situations of $\omega$ between $1$ and $n$, $s_n(x)$ the number of those essential return situations which are actual essential return times, $d_n(x)$ the number of those essential returns which have deep essential return with depth above threshold $\Theta \geq \Delta$ whose upper bound is $\frac{\beta+\log 4}{s-1}\frac{n}{\Theta}$ (each return will be followed by a bound period of length greater than $\frac{\beta+\log 4}{s-1}\frac{1}{\Theta}$, Lemma~\ref{lemm}). But the essential return situation is a chopping time and it can be a return time or an escape time for every chopping component, so $r_n(x)-s_n(x)$ is the exact number of escaping times of $\omega$.

Given an integer $d$ with $$0\leq d\leq \frac{\beta+\log 4}{s-1}\frac{n}{\Theta},$$ 
an integer $r$ with $d\leq r\leq n$ and $d$ integers $m_1,\ldots, m_d\ge\Theta$, we define
$$A^{r,d}_{m_1,\ldots,m_d}(n)=
\left\{
         \begin{array}{lll}
           
           x\in I:\parbox{0.55\columnwidth}{$r_n(x)=r$, $d_n(x)=d$, and the depth of the $j$-th deep essential return is $m_j$ for $j\in\{1,\dots,d\}$}
           
         \end{array}
\right\}$$
\begin{proposition}\label{proess}
We have 
$$
\Big|A^{r,d}_{m_1,\ldots,m_d}(n)\Big|\leq\binom{r}{d} \exp\left(-\left(1-\beta\frac{s+5}{\beta+\log\lambda_c}\right)\sum\limits_{j=1}^{d}m_j\right).$$
\end{proposition}
\begin{proof}
Take $n\in\mathbb{N}$ and $\omega_0\in\mathcal{P}_0$. For every $\omega\in\mathcal{P}_n$ with $\omega\subseteq\omega_0$ and $r_n(\omega)=r$, let $\omega_i$ indicate the element of the partition $\mathcal{P}_{t_i}$ containing $\omega$ where $t_i$ is the $i$-th return situation. We have $0\leq t_1\leq t_2\leq\ldots\leq t_r\leq n$ and $\omega_0\supseteq\omega_1\supseteq\ldots\supseteq\omega_r=\omega$.

For each $i\in\{0,\ldots,n\}$ we define
$$
\mathcal{Q}_i=\{\omega_i: \omega\in\mathcal{P}_n\mbox{ and }\omega\subseteq\omega_0\mbox{ with }r_n(\omega)=r\}.
$$
Fix $d$ integers $0\leq u_1\leq u_2\leq\ldots\leq u_d\leq r$ with $u_j$ indicating that the $j$-th deep essential return occurs in the $u_j$-th essential return situation, i.e. $t_{u_j}$ is the $j$-th deep essential return time. Now we just consider those elements of the partition $\mathcal{P}_n$ which are subsets of $\omega_0$ with $r$ times essential return situations and at its $u_j$-th essential return situation its $j$-th deep essential return time occurs with depth $m_j$. And to do that, set $V(0)=\omega_0$. For $i\leq r$ we define $V(i)$ recursively. Suppose that $V(i-1)$ is already defined. If $u_{j-1}<i<u_j$, we set
$$V(i)=\bigcup\limits_{\omega\in\mathcal{Q}_i}\omega\cap f^{-t_i}(I\setminus U_{\Theta})\cap V(i-1),$$
and if $i=u_j$ we set
$$V(i)=\bigcup\limits_{\omega\in\mathcal{Q}_i}\omega\cap f^{-t_i}(I_{m_j}\cap I_{-m_j})\cap V(i-1).$$
Observe that for every $i\in\{1,\ldots,r\}$ we have $\frac{|V(i)|}{|V(i-1)|}\leq 1$, but we find a better estimate for $\frac{|V(u_j)|}{|V(u_j-1)|}$. Take $\omega_{u_j}\in V(u_j)\cap\mathcal{Q}_{u_j}$ and $\omega_{u_j-1}\in V(u_j-1)\cap\mathcal{Q}_{u_j-1}$. 
We consider two situations depending on whether $t_{u_j-1}$ is an escaping situation or an essential return.
\begin{enumerate} 
\item First suppose that $t_{u_j-1}$ is an essential return with depth $m$. Then 
\begin{align*}
\frac{\left|\omega_{u_j}\right|}{\left|\omega_{u_j-1}\right|}&\leq \frac{\left|\omega_{u_j}\right|}{\left|\hat{\omega}_{u_j-1}\right|},\quad\mbox{where }\hat{\omega}_{u_j-1}=\omega_{u_j-1}\cap f^{-t_{u_j}}(U_{\Delta_0})\mbox{ and }\Delta_0=\left\lceil\beta\frac{s+2}{\beta+\log\lambda_c}\Delta\right\rceil\\
&\leq C_2\frac{\left|f^{t_{u_j}}(\omega_{u_j})\right|}{\left|f^{t_{u_j-1}}(\hat{\omega}_{u_j-1})\right|},\quad\mbox{by Mean Value Theorem and Proposition }\ref{distortion}\\
&\leq C_2\frac{2e^{-m_j}}{\left|f^{t_{u_j-1}}(\hat{\omega}_{u_j-1})\right|},\quad
\mbox{by definition of }\omega_{u_j}.
\end{align*}
we consider two cases,
\begin{enumerate}
\item if $\hat{\omega}_{u_i-1}=\omega_{u_i-1}$, then by Lemma \ref{lem3} part (3b) 
$$
|f^{t_{u_i}}(\hat{\omega}_{u_i-1})|\geq e^{-\beta\frac{s+3}{\beta+\log\lambda_c}m}.
$$
\item if $\hat{\omega}_{u_i-1}\neq\omega_{u_i-1}$, then $f^{t_{u_i}}(\hat{\omega}_{u_i-1})$ has a point outside $U_{\Delta_0}$. Since we are assuming $\omega_{u_i}\neq 0$ and $\omega_{u_i}\subseteq\hat{\omega}_{u_i-1}$,
 therefore $f^{t_{u_i}}(\hat{\omega}_{u_i-1})$ has a point inside $U_{\Delta}$ and then
\begin{equation*}
|f^{t_{u_i}}(\hat{\omega}_{u_i-1})|\geq e^{-\beta\frac{s+2}{\beta+\log\lambda_c}\Delta}-e^{-\Delta}\geq e^{-\beta\frac{s+3}{\beta+\log\lambda_c}\Delta}.
\end{equation*}
On the other hand, we have $m\geq\Delta$ which implies $e^{-\beta\frac{s+3}{\beta+\log\lambda_c}\Delta}\geq e^{-\beta\frac{s+3}{\beta+\log\lambda_c}m}$. Hence
$$
|f^{t_{u_i}}(\hat{\omega}_{u_i-1})|\geq e^{-\beta\frac{s+3}{\beta+\log\lambda_c}m}.
$$
\end{enumerate}
Consequently, in both cases we have 
$$
\frac{|\omega_{u_i}|}{|\omega_{u_i-1}|}\leq C_2\frac{2e^{-\rho_i}}{e^{-\beta\frac{s+3}{\beta+\log\lambda_c}m}}.
$$
Note that when $u_{j}-1=u_{j-1}$, then $m=m_{u_{j-1}}\geq\Theta$. On the other hand, if $u_j-1>u_{j-1}$, then $t_{u_j-1}$ is an essential return with depth $m\leq\Theta\leq m_{j-1}$. In both cases
\begin{equation}\label{eq110}
\frac{\left|\omega_{u_j}\right|}{\left|\omega_{u_j-1}\right|}\leq 2C_2\frac{e^{-m_j}}{e^{-\beta\frac{s+3}{\beta+\log\lambda_c}m_{j-1}}}~.
\end{equation}
\item now suppose $t_{r_i-1}$ is an escape situation. We have a same estimate as in \eqref{eq110}. 
We only use Lemma \ref{lemi} instead of Lemma \ref{lem3} in (a).
\end{enumerate}
Now it follows that
\begin{align*}
|V(u_j)|&=\sum\limits_{\omega_{u_j}\in\mathcal{Q}_{u_j}}\frac{|\omega_{u_j}|}{|\omega_{u_j-1}|}|\omega_{u_j-1}|\\
&\leq 2C_2e^{-m_j}e^{\beta\frac{s+3}{\beta+\log\lambda_c}m_{j-1}}\sum\limits_{\omega_{u_j}\in\mathcal{Q}_{u_j}}|\omega_{u_j-1}|\\
&\leq 2C_2e^{-m_j}e^{\beta\frac{s+3}{\beta+\log\lambda_c}m_{j-1}}|V(u_j-1)|.
\end{align*}
Therefore we have
\begin{equation}\label{eq220}
|V(u)|=(2C_2)^d \exp\left(-\left(1-\beta\frac{s+3}{\beta+\log\lambda_c}\right)\sum\limits_{j=1}^{d}m_j\right)e^{\beta\frac{s+3}{\beta+\log\lambda_c}m_0}|V(0)|,
\end{equation}
where $m_0=0$ if $\omega_0$ be $(\delta, 1]$ or $[-1,-\delta)$, and $m_0=|\eta|$ if $\omega_0=I_{\eta, K}$ for some $|\eta|\geq\Delta$ and $k\in\{1,2,\ldots,\eta^2\}$. 
Let $B=\beta\frac{s+3}{\beta+\log\lambda_c}$, and from \eqref{eq220} we see that 
\begin{align*}
\Big|A^{r,d}_{m_1,\ldots,m_d}(n)\Big|&\leq (2C_2)^d\binom{r}{d} \exp\left(-(1-B)\sum\limits_{j=1}^{d}m_j\right)\sum\limits_{\omega_0\in\mathcal{P}_0}e^{B m_0}|\omega_0|\\
&\leq(2C_2)^d\binom{r}{d} \exp\left(-(1-B)\sum\limits_{j=1}^{d}m_j\right)\left(2(1-\delta)+\sum\limits_{|m_0|\geq\Delta}e^{B|m_0|}e^{-|m_0|}\right)\\
&\leq 3(2C_2)^d\binom{r}{d} \exp\left(-(1-B)\sum\limits_{j=1}^{d}m_j\right)\\
&\leq \binom{r}{d} \exp\left(-\left(1-\beta\frac{s+5}{\beta+\log\lambda_c}\right)\sum\limits_{j=1}^{d}m_j\right),
\end{align*}
the last inequality holds since $d\Theta\leq\sum\limits_{j=1}^{d}m_j$ and we can chose $\Theta$ sufficiently large. 
\end{proof}
As a corollary we can find the probability of the event that $j$-th deep essential of its elements reach depth $m$, i.e.
\begin{equation}\label{eq330}
A^{r,d}_{m,j}(n)=
\left\{
         \begin{array}{ll}
           x\in I:\parbox{0.55\columnwidth}{$r_n(x)=r, d_n(x)=d$, and the depth of the $j$-th deep essential return is $m$}
         \end{array}
\right\}.
\end{equation}
\begin{corollary}\label{pro1ess}
If $\Theta$ is large enough, then
$$\left|A^{r,d}_{m,j}(n)\right|\leq \binom{r}{d} e^{-(1-\beta\frac{s+5}{\beta+\log\lambda_c})m}.$$
\end{corollary}
\begin{proof}
Note that $A^{r,d}_{m,j}(n)=\bigcup\limits_{m_i\geq\Theta\atop i\neq j}A^{r,d}_{m_1,\ldots,m_{j-1},m,m_{j+1},\ldots,m_d}(n)$. By Proposition \ref {proess},
\begin{align*}
\Big|A^{r,d}_{m,j}(n)\Big|&\leq \binom{r}{d}e^{-(1-\beta\frac{s+5}{\beta+\log\lambda_c})m}\left(\sum\limits_{\eta=\Theta}^{\infty}e^{-(1-\beta\frac{s+5}{\beta+\log\lambda_c})\eta}\right)^{d-1}\\
&\leq \binom{r}{d}e^{-(1-\beta\frac{s+5}{\beta+\log\lambda_c})m},
\end{align*}
for $\Theta$ large enough such that $\sum\limits_{\eta=\Theta}^{\infty}e^{-(1-\beta\frac{s+5}{\beta+\log\lambda_c})\eta}\leq 1$.
\end{proof}

\section{The measure of the tail set}\label{se.measure.tail}
Here we finish the proof of Theorem~\ref{maintheorem}. First we check the non-uniform expansion and later the slow recurrence to the critical set.

\subsection{Nonuniform expansion}

Assume that $n$ is a fixed large integer. We define
$$\mathsf{E}_1(n)=\{x\in I: \exists i\in\{1,\ldots,n\}~\mbox{such~that}~|f^i(x)|\leq e^{-\alpha n}\}.$$
Take $x\in I\setminus \mathsf{E}_1(n)$ and suppose that $z_1,\ldots,z_{\gamma}$ are return times of $x$, either essential or inessential up to time $n$. Let $p_i$ be the associated bound period originated by return time $z_i$. Set $z_0=0$ and if $|x|\geq\delta$, set $p_0=0$. 
We define $q_i=z_{i+1}-(z_i+p_i)$ for $i=0,\ldots,\gamma-1$ and
$$
q_{\gamma}=\left\{
\begin{array}{ll}
                   0&\mbox{if}~n<z_{\gamma}+p_{\gamma},\\
                   n-(z_{\gamma}+p_{\gamma})&\mbox{if}~n\geq z_{\gamma}+p_{\gamma}.
\end{array}
\right.
$$
Take $c=\frac{\log\lambda_c}{s+1}-\beta-s\alpha$. If $n\geq z_{\gamma}+p_{\gamma}$, then by Lemmas \ref{lem1} and \ref{lemm} we get
\begin{align*}
(f^n)'(x)&=\prod\limits_{i=0}^{\gamma}(f^{q_i})'(f^{z_i+p_i}(x))(f^{p_i})'(f^{z_i}(x))\\
&\geq c(\delta) e^{\log\lambda\sum\limits_{i=0}^{\gamma}q_i}e^{(\frac{\log\lambda_c}{s+1}-\beta)\sum\limits_{i=0}^{\gamma}p_i}\\
&\geq c(\delta) e^{(\frac{\log\lambda_c}{s+1}-\beta)n},\quad\log\lambda>\log\lambda_c>\frac{\log\lambda_c}{s+1}-\beta\\
&\geq c(\delta) e^{s\alpha n}e^{c n},\quad n\mbox{ big enough such that }c(\delta) e^{s\alpha n}>1\\
&\geq e^{c n}.
\end{align*}
If $n\leq z_{\gamma}+p_{\gamma}$ then by the same lemmas and the fact that $x\notin \mathsf{E}_1(n)$, it follows 
\begin{align*}
(f^n)'(x)&=f'(f^{z_{\gamma}}(x))(f^{n-z_{\gamma}-1})'(f^{z_{\gamma}}(x))\prod\limits_{i=0}^{\gamma-1}(f^{q_i})'(f^{z_i+p_i}(x))(f^{p_i})'(f^{z_i}(x))\\
&\geq \frac{K_2}{C_0}e^{-\alpha n(s-1)}e^{\log\lambda(n-z_{\gamma}-1)}e^{\log\lambda\sum\limits_{i=0}^{\gamma-1}q_i}e^{(\frac{\log\lambda_c}{s+1}-\beta)\sum\limits_{i=0}^{\gamma-1}p_i},\quad\mbox{Lemmas }\ref{lem1}\mbox{ and }\ref{lemm}\\
&\geq \frac{K_2}{C_0}e^{-\alpha n(s-1)}e^{(\frac{\log\lambda_c}{s+1}-\beta)(n-1)},\quad n\mbox{ large such that }\frac{K_2}{C_0}e^{-\alpha n(s-1)}e^{-\alpha s-c}\geq e^{-\alpha sn}\\
&\geq e^{-\alpha sn}e^{(\alpha s+c)n}\\
&\geq e^{c n}.
\end{align*}
Therefore we have proved that if $x\notin \mathsf{E}_1(n)$, then $(f^n)'(x)\geq e^{c n}$ for some $c>0$.
We will show that
$$|\mathsf{E}_1(n)|\leq e^{-\tau_1 n}, \hspace{1cm}\forall n\geq N_1,$$
for some constant $\tau_1(\alpha,\beta)$ and an integer $N_1(\Delta,\tau_1)$.

We can take $\Theta=\Delta$ in \eqref{eq330} and define
$$A^{r,d}_{m}(n)=
\left\{
         \begin{array}{ll}
           x\in I:\parbox{0.55\columnwidth}{$r_n(x)=r, d_n(x)=d$, and there is an essential return with depth $m$}
         \end{array}
\right\},$$
for fixed $r, d$ ($d\leq r\leq n$) and $m\geq\Delta$, and
$$A_{m}(n)=\Big\{x\in I: \exists t\leq n \mbox{ such that }t\mbox{ is an essential return and }|f^t(x)|\in I_{m}\Big\},$$
for fixed $n$ and $m\geq\Delta$. Since $A^{r,d}_{m}(n)=\bigcup\limits_{j=1}^{d}A^{r,d}_{m,j}(n)$, by Corollary \ref{pro1ess}
\begin{align}\label{1}
\left|A^{r,d}_{m}(n)\right|\leq\sum\limits_{j=1}^{d}\left|A^{r,d}_{m,j}(n)\right|\leq d\binom{r}{d}e^{-(1-\beta\frac{s+5}{\beta+\log\lambda_c})m}.
\end{align}
Since $A_{m}(n)=\bigcup\limits_{d=1}^{\frac{\beta+\log 4}{s-1}\frac{n}{\Delta}}\bigcup\limits_{r=d}^{n}A^{r,d}_{m}(n)$, then by \eqref{1}
\begin{align*}
|A_{m}(n)|&\leq\sum\limits_{d=1}^{\frac{\beta+\log 4}{s-1}\frac{n}{\Delta}}\sum\limits_{r=d}^{n}\left|A^{r,d}_{m}(n)\right|
\leq\sum\limits_{d=1}^{\frac{\beta+\log 4}{s-1}\frac{n}{\Delta}}\sum\limits_{r=d}^{n}d\binom{r}{d}e^{-(1-\beta\frac{s+5}{\beta+\log\lambda_c})m}\\
&\leq e^{-(1-\beta\frac{s+5}{\beta+\log\lambda_c})m}\sum\limits_{d=1}^{\frac{\beta+\log 4}{s-1}\frac{n}{\Delta}}d\sum\limits_{r=d}^{n}\binom{n}{d}
\leq n e^{-(1-\beta\frac{s+5}{\beta+\log\lambda_c})m}\sum\limits_{d=1}^{\frac{\beta+\log 4}{s-1}\frac{n}{\Delta}}d\binom{n}{d}\\
&\leq \frac{n^3}{\Delta}\binom{n}{\frac{\beta+\log 4}{s-1}\frac{n}{\Delta}}e^{-(1-\beta\frac{s+5}{\beta+\log\lambda_c})m}.
\end{align*}
Take $R=\frac{\beta+\log 4}{s-1}\frac{n}{\Delta}$. The Stirling Formula
$$\sqrt{2\pi n}n^n e^{-n}\leq n!\leq\sqrt{2\pi n}n^n e^{-n}\Big(1+\frac{1}{4n}\Big),$$
implies that
$$\binom {n}{R}\leq \const\frac{n^n}{(n-R)^{n-R}(R)^{R}}.$$
So
$$\binom{n}{R}\leq \const\Bigg(\Bigg(1+\frac{\frac{R}{n}}{1-\frac{R}{n}}\Bigg)\Bigg(1+\frac{1-\frac{R}{n}}{\frac{R}{n}}\Bigg)^{\frac{\frac{R}{n}}{1-\frac{R}{n}}}\Bigg)^{(1-\frac{R}{n})n}.
$$
By Taking  $h(\Delta)=(1-\frac{R}{n})\log\Bigg(\Bigg(1+\frac{\frac{R}{n}}{1-\frac{R}{n}}\Bigg)\Bigg(1+\frac{1-\frac{R}{n}}{\frac{R}{n}}\Bigg)^{\frac{\frac{R}{n}}{1-\frac{R}{n}}}\Bigg)$ we have that $h(\Delta)\to 0$ when $\Delta\to +\infty$ and
$$\binom{n}{R}\leq\const~e^{h(\Delta)n}.$$
Since the depths of inessential and bound returns are less than the depth of the essential returns preceding them, for $n$ such that $\alpha n\geq\Delta$
$$\mathsf{E}_1(n)=\Big\{x\in I: |f^i(x)|<e^{-\alpha n}\mbox{ for some }i\in\{1,\cdots,n\}\Big\}\subseteq\bigcup\limits_{m=\alpha n}^{+\infty}A_{m}(n).$$
Let us take $\Delta$ large enough such that $h(\Delta)\leq (1-\beta\frac{s+5}{\beta+\log\lambda_c})\frac{\alpha}{2}$. Then
\begin{align*}
|\mathsf{E}_1(n)|
&\leq \const~\frac{n^3}{\Delta}e^{h(\Delta)n}\sum\limits_{m=\alpha n}^{+\infty}e^{-(1-\beta\frac{s+5}{\beta+\log\lambda_c})m}\\
&\leq \const~\frac{n^3}{\Delta}e^{h(\Delta)n}e^{-(1-\beta\frac{s+5}{\beta+\log\lambda_c})\alpha n}\\
&\leq \const~\frac{n^3}{\Delta}e^{-(1-\beta\frac{s+5}{\beta+\log\lambda_c})\frac{\alpha n}{2}}\\
&\leq \const~\frac{n^3}{\Delta}e^{-2\tau_1 n}\\
&\leq e^{-\tau_1 n},
\end{align*}
where $\tau_1=(1-\beta\frac{s+5}{\beta+\log\lambda_c})\frac{\alpha}{4}$ and $n$ is large enough such that
$$\const~\frac{n^3}{\Delta}e^{-\tau_1 n}\leq 1.$$
Therefore, for large $n$, say $n>N_1$, we have $(f^n)'(x)\geq e^{c n}$ for every $x\in I$ except for the points in $\mathsf{E}_1(n)$. Now we exclude the points which do not verify (\ref{eq:1}), i.e.
$$\mathsf{E}_1=\bigcap\limits_{k\geq N}\bigcup\limits_{n\geq k}\mathsf{E}_1(n).$$
On the other hand, since for every $k\geq N_1$
$$\sum\limits_{n\geq k}|\mathsf{E}_1(n)|\leq \const~e^{-\tau_1 k},$$
Borel-Cantelli Lemma implies that $|\mathsf{E}_1|=0$. As a result, (\ref{eq:1}) holds on the full Lebesgue measure set $I\setminus \mathsf{E}_1$. Note that $\{x\in I: \mathcal{E}(x)>k\}\subseteq \bigcup\limits_{n\geq k}\mathsf{E}_1(n)$. Thus for $k\geq N_1$
$$\Big|\{x\in I: \mathcal{E}(x)>k\}\Big|\leq \const~e^{-\tau_1 k}.$$
Therefore, there exists $C=C(N_1,\tau_1)$ such that for all $n\in\mathbb{N}$
$$\Big|\{x\in I: \mathcal{E}(x)>n\}\Big|\leq Ce^{-\tau_1 n}.$$

\subsection{Slow recurrence to the critical set}
We define for every $x$ and $n$, 
$$T_n(x)=\frac 1n\sum\limits_{j=0}^{n-1}-\log (\d_{\delta}(f^j(x),0)),$$
where $\delta=e^{-\Theta}$. Note that the only points of the orbit of $x$ that contribute to the sum are those with deep return and its depth is above the threshold $\Theta\geq\Delta$. 
According to the basic idea expressed in Section~\ref{se.retrndeoths}, in order to obtain a bound for $T_n$ we only need to find an upper bound for the sum of inessential and bound returns depths occurring between two consecutive essential returns. Using Lemmas \ref{1lem} and \ref{2lem}, it can be seen that if $t$ is an essential return time with depth $|m|$ then sum of its depth and depth of all inessential returns and bound returns before the next essential return is less than $(C_3+C_3C_4)|m|$. 
Thus, if we define $F_n(x)=\sum\limits_{j=1}^{d}m_j$ such that $d$ be the number of essential returns of $x$ with depth above $\Theta$ up to time $n$ and $m_j$'s are their respective depth, then it follows that
\begin{equation}\label{TF}
T_n(x)\leq \frac{C_5}{n}F_n(x),\qquad C_5=C_3+C_3C_4.
\end{equation}
We define for all $n\in\mathbb{N}$,
$$\mathsf{E}_2(n)=\big\{x\in I: T_n(x)>\epsilon\big\}.$$
From (\ref{TF}), it can be concluded that  
$$\big|\mathsf{E}_2(n)\big|\leq\left|\left\{x\in I: F_n(x)\geq\frac{n\epsilon}{C_5}\right\}\right|.$$
In order to complete the proof of Theorem \ref{maintheorem}, we show that
$$\left|\left\{x\in I: F_n(x)\geq\frac{n\epsilon}{C_5}\right\}\right|\leq e^{-\tau_2  n}.$$
\begin{lemma}
Let $0<t\leq(1-\beta\frac{s+5}{\beta+\log\lambda_c})\frac{1}{3}$. For $\Theta$ large enough,
$$\mathbb{E}\left(e^{tF_n}\right)\leq \const~\frac{n^2}{\Theta}e^{h(\Theta)n},$$
where $\mathbb{E}$ is the mathematical expectation. Moreover $h(\Theta)\to 0$ when $\Theta\to +\infty.$
\end{lemma}
\begin{proof}
\begin{align*}
\mathbb{E}\left(e^{tF_n}\right)&=\mathbb{E}\left(e^{t\sum_{j=1}^{d}m_j}\right)=\sum\limits_{r,d,(m_1,\ldots,m_d)}e^{t\sum_{j=1}^{d}m_j}\Big|A_{m_1,\ldots,m_d}^{r,d}(n)\Big|\\
&\leq \sum\limits_{r,s,(m_1,\ldots,m_n)}e^{t\sum_{j=1}^{d}m_j}\binom{r}{d}e^{-3t\sum_{j=1}^{d}m_j}\\
&\leq \sum\limits_{r,d,Q}\binom{r}{d}\zeta(d,Q)e^{-2tQ}
\end{align*}
where $\zeta(d,Q)$ is the number of integer solutions of the equation $x_1+\ldots+x_d=Q$ with $x_j\geq \Theta$ for all $j$. So
$$\zeta(d,Q)\leq \#\{\mbox{solutions~of}~x_1+\ldots+x_d=Q,~x_j\in\mathbb{N}_0\}=\binom{Q+d-1}{d-1}.$$
By Stirling Formula we have
\begin{align*}
\binom{Q+d-1}{d-1}&\leq\const\frac{(Q+d-1)^{Q+d-1}}{Q^Q(d-1)^{d-1}}\\
&\leq \Bigg(\const^{\frac 1Q}\Big(1+\frac{d-1}{Q}\Big)\Big(1+\frac{Q}{d-1}\Big)^{\frac{d-1}{Q}}\Bigg)^Q.
\end{align*}
Since $d\Theta\leq Q$, each factor in the last expression can be made arbitrarily close to $1$ by taking $\Theta$ large enough. Therefore
$$\binom{Q+d-1}{d-1}\leq e^{tQ},$$
and
$$\mathbb{E}\left(e^{tF_n}\right)\leq \sum\limits_{r,d,Q}\binom{r}{d}e^{tQ}e^{-2tQ}\leq\sum\limits_{r,d,Q}\binom{r}{d}e^{-tQ}\leq\sum\limits_{r,d}\binom{r}{d}.$$
Now
\begin{align*}
\sum\limits_{r,d}\binom{r}{d}&\leq\sum\limits_{d=1}^{\frac{\beta+\log 4}{s-1}\frac{n}{\Theta}}\sum\limits_{r=d}^{n}\binom{r}{d}\\
&\leq n\sum\limits_{d=1}^{\frac{\beta+\log 4}{s-1}\frac{n}{\Theta}}\binom{n}{d}\\
&\leq n\sum\limits_{d=1}^{\frac{\beta+\log 4}{s-1}\frac{n}{\Theta}}\binom{n}{\frac{\beta+\log 4}{s-1}\frac{n}{\Theta}}\\
&\leq \frac{\beta+\log 4}{s-1}\frac{n^2}{\Theta}\binom{n}{\frac{\beta+\log 4}{s-1}\frac{n}{\Theta}}.
\end{align*}
Using again the Stirling Formula, we have
$$\mathbb{E}\left(e^{tF_n}\right)\leq\const~\frac{n^2}{\Theta}e^{h(\Theta)n}$$
where $h(\theta)\to 0$ when $\Theta\to +\infty$.
\end{proof}
If we take $t=(1-\beta\frac{s+5}{\beta+\log\lambda_c})\frac{1}{3}$ and $\Theta$ large enough such that $2\tau_2=\frac{t\epsilon}{C_5}-h(\Theta)>0$
\begin{align*}
\left|\left\{x\in I: F_n(x)>\frac{n\epsilon}{C_5}\right\}\right|&\leq e^{-t\frac{n\epsilon}{C_5}}\mathbb{E}\left(e^{tF_n}\right),\quad\mbox{Chebyshev's inequality}\\
&\leq \const\frac{n^2}{\Theta}e^{-t\frac{n\epsilon}{C_5}}e^{h(\Theta)n}\\
&\leq \const\frac{n^2}{\Theta}e^{-2\tau_2 n}\\
&\leq e^{-\tau_2n}
\end{align*}
for big enough $n$, say $n\geq N_2$, such that $\const\frac{n^2}{\Theta}e^{-\tau_2n}<1$. Therefore, there exists $C=C(N_2,\tau_2)$ such that for all $n\in\mathbb{N}$
$$\Big|\{x\in I: \mathcal{R}(x)>n\}\Big|\leq Ce^{-\tau_2 n}.$$

\appendix 
\section{Statistical properties}\label{ap.statistical}\label{ap.a}

Here we introduce the precise formulations of the statistical notions used in Corollary~\ref{co.statistical}. Let $\mathcal H$  denote Banach space of H\"older continuous functions, for some fixed exponent $\gamma>0$. We consider, for some probability measure $\mu$,  the Banach space of essentially bounded functions $L^\infty(\mu)$.

\subsection{Decay of correlations}
\label{ss.A0} We define the \emph{correlation} of $\varphi\in \mathcal H$ and  \( \psi\in L^\infty(\mu) \)  as
\[
\cv_\mu (\varphi,\psi\circ f^n):=\left|\int \varphi\, (\psi\circ f^n)\, d\mu-\int  \varphi\, d\mu\int
\psi\, d\mu\right|.
\]
We say that we have \emph{exponential decay
of correlations} for H\"older observables  \emph{against}
observables in $L^\infty(\mu)$ if there are $C>0$ and $\tau>0$ such that for all $\varphi\in\mathcal H$, 
$\psi\in L^\infty(\mu)$ and $n\ge 1$
 $$\cv_\mu(\varphi,\psi\circ f^n)\le C\|\varphi\|\cdot\|\psi\|_\infty e^{-\tau n}.$$

\subsection{Large deviations}
\label{ss.A1}  Given  $\varphi\in\mathcal H$ and
 \( \epsilon>0 \)
we define the \emph{large deviation} of $\varphi$ at time~$n$
~as
\[
\ld_{\mu}(\varphi,\epsilon, n):=\mu\left(\left|\frac1 n \sum_{i=0}^{n-1}
\varphi\circ f^i-\int\varphi d\mu \right|>\epsilon\right).
\]
By Birkhoff's ergodic theorem  the quantity \( \ld_{\mu}(\varphi,\epsilon,n) \to 0 \), as
\( n\to \infty \). We say that we have \emph{exponential large deviateions} for H\"older observables   if there are $C=C(\epsilon)>0$ and $\tau>0$ such that for all $\varphi\in\mathcal H$ and $n\ge 1$
 $$\ld_{\mu}(\varphi,\epsilon, n)\le C\|\varphi\| e^{-\tau n}.$$

\subsection{Central Limit Theorem}
\label{ss.A1C}
Let $\varphi\in\mathcal H$ be such that $\int \varphi d\mu=0$. Then
\begin{equation}\label{eq.sigma}
\sigma^2=\lim_{n\to\infty}\frac1n \int \left(\sum_{i=0}^{n-1} \varphi\circ f^i\right)^2d\mu\geq0
\end{equation}
is well defined. We say the {\em Central Limit Theorem} holds for $\varphi$ if for all $a\in\mathbb R$
\[
\mu\left(\left\{x: \frac1{\sqrt n} \sum_{i=0}^{n-1} \varphi\circ f^i(x)\leq a\right\}\right)\rightarrow \int_{-\infty}^a \frac1{\sigma\sqrt{2\pi}}\text e^{-\frac{x^2}{2\sigma^2}}dx, \text{ as $n\to\infty$},
\]
whenever $\sigma^2>0$.
Additionally, $\sigma^2=0$ if and only if $\varphi$ is a \emph{coboundary}  (\( \varphi\neq \psi\circ f - \psi \) for any \(
\psi \in L^2(\mu)\)).

\subsection{Local Limit Theorem}
\label{ss.A2}
\label{aperiodic} A function $\varphi:I\to\mathbb R$ is said to be \emph{periodic} if there exist $\rho\in\mathbb R$, a measurable 
function $\psi:I\to\mathbb R$, $\lambda>0$, and
$q:I\to\mathbb Z$, such that $$\varphi=\rho+\psi-\psi\circ f+\lambda q$$ almost everywhere. Otherwise, it is said to be \emph{aperiodic}.

Let
$\varphi \in \mathcal H$ be such that $\int \varphi d\mu=0$ and $\sigma^2$ be as in \eqref{eq.sigma}. Assume that  
$\varphi$ is aperiodic (which implies that $\sigma^2>0$). 
We say that the {\em Local Limit Theorem} holds for  $\varphi$ if for any bounded interval 
$J\subset \mathbb R$,  for any real sequence $\{k_n\}_{n\in\mathbb N}$ with 
$k_n/n \to \kappa \in\mathbb R$, for any $u\in\mathcal H$, for any measurable $v:I\to\mathbb R$ we have
\[
\sqrt n \mu\left(\left\{x\in M:\; \sum_{i=0}^{n-1} \varphi\circ f^i(x)\in J+k_n+u(x)+v(f^nx)\right\}\right)\to m(J) \frac{\mbox{e}^{-\frac{\kappa^2}{2\sigma^2}}}{\sigma\sqrt{2\pi}}.
\]

\subsection{Berry-Esseen Inequality}
\label{ss.A3}
If $f$ admits a Young tower of base $\Delta_0\subset I$ and return time function $R$, then for any 
$\varphi:I\to\mathbb R$ define $\varphi_{\Delta_0}:\Delta_0\to\mathbb R$ by
\[
\varphi_{\Delta_0}(x)=\sum_{i=0}^{ R(x)-1}\varphi (f^ix).
\]

Let $\varphi\in\mathcal H$ be such that $\int \varphi d\mu=0$ and $\sigma^2$ be as in \eqref{eq.sigma}. 
Assume that $\sigma^2>0$ and that there exists $0<\delta\leq 1$ such that 
$\int |\varphi_{\Delta_0}|^2\chi_{|\varphi_{\Delta_0}|> z}d\mu\le \const z^{-\delta}$, for large $z$. If $\delta=1$, 
assume also that  $\int |\varphi_{\Delta_0}|^3\chi_{|\varphi_{\Delta_0}|\leq z}d\mu$ is bounded. 
We say that {\em Berry-Esseen Inequality} holds for $\varphi$ if there exists $C>0$ such that for all 
$n\in\mathbb N$ and $a\in\mathbb R$ we have
\[
\left|\mu\left(\left\{x: \frac1{\sqrt n} \sum_{i=0}^{n-1} \varphi\circ f^i(x)\leq a\right\}\right) - \int_{-\infty}^a \frac1{\sigma\sqrt{2\pi}}\text e^{-\frac{x^2}{2\sigma^2}}dx\right|\leq \frac{C}{n^{\delta/2}}.
\]

\subsection{Almost Sure Invariance Principle}
\label{ss.A4}
Given $d\ge 1$ and a H\"older continuous $\varphi\colon M\to \mathbb R^d$ with $\int \varphi d\mu =0$, we
denote $$S_n=\sum_{i=0}^{n-1}\varphi\circ f^i, \quad\text{for each $n\ge 1$.}$$
We say that $\varphi$ satisfies an \emph{Almost Sure Invariance Principle}
if there exists $\lambda > 0$ and  a probability
space supporting a sequence of random variables $\{S^*_n\}_n$ (which can be $\{S_n\}_n$ in the $d=1$ case)
and a $d$-dimensional Brownian motion
$W(t)$ such that
\begin{enumerate}
\item  $\{S_n\}_n $ and $ \{S^*_n\}_n $ are equally distributed;
\item  $S^*_n = W(n) + O(n^{1/2-\lambda})$, as $n\to\infty$,
almost everywhere.
\end{enumerate}


\end{document}